\documentclass{jams-l}

\usepackage[utf8]{inputenc}
\usepackage{amsmath, amssymb, amsthm}
\usepackage{mathrsfs} %script font
\usepackage{mathtools}
\usepackage{bm} % bold symbols in math mode
\usepackage{hyperref} %hyperlinks and cross-referencing
\usepackage{thmtools} % extensions to theorem environments
\usepackage{faktor} %factor structures (numerator/denominator)
\usepackage{bbm}
\usepackage{setspace}

\usepackage[alphabetic]{amsrefs} % bibliography style

\newcommand{\norm}[1]{\left\lVert#1\right\rVert}

\DeclareMathOperator{\End}{End}
\DeclareMathOperator{\Cl}{C\ell}

\DeclareMathOperator{\dirac}{\mathbb{D}}

\DeclareMathOperator{\J}{\mathbb{J}}
\DeclareMathOperator{\Z}{\mathbb{Z}}
\DeclareMathOperator{\R}{\mathbb{R}}
\DeclareMathOperator{\T}{\mathbb{T}}
\DeclareMathOperator{\C}{\mathbb{C}}
\DeclareMathOperator{\h}{\mathfrak{H}}
\DeclareMathOperator{\B}{\mathfrak{B}}
\DeclareMathOperator{\Q}{\mathbb{Q}}
\DeclareMathOperator{\A}{\mathbb{A}}
\DeclareMathOperator{\modcurve}{\Gamma\backslash G/K}

\usepackage{tikz-cd} % (commutative) diagrams

\copyrightinfo{2021}{Adrienne Sands}

\newtheorem{theorem}{Theorem}[section]
\newtheorem{lemma}[theorem]{Lemma}

\theoremstyle{definition}

\newtheorem{corollary}{Corollary}[theorem]

\theoremstyle{remark}

\numberwithin{equation}{section}

\begin{document}

\title[Automorphic Hamiltonians]{Automorphic Hamiltonians, Epstein zeta functions, and Kronecker limit formulas}


%    author one information
\author{Adrienne Sands}
\address{School of Mathematics, University of Minnesota, Minneapolis, MN 55455}
\curraddr{MIT Lincoln Laboratory, 244 Wood St, Lexington, MA 02421}
\email{adrienne.sands@ll.mit.edu}
\thanks{This material is based upon work supported by the National Science Foundation Graduate Research Fellowship Program under Grant \#39202.}
\dedicatory{This manuscript is dedicated to my advisor, Paul Garrett.}

%    author two information
%\author{}
%\address{}
%\curraddr{}
%\email{}
%\thanks{}

%    \subjclass is required.
\subjclass[2022]{Number Theory}

\date{\today}

%    Abstract is required.
\begin{abstract}
First, we recount a history of how certain methods using natural self-adjoint operators have, thus far, failed to prove the Riemann Hypothesis. In Section \ref{Background}, we set the analytical context necessary to have genuine proofs in later sections, rather than attractive heuristics. {In Section \ref{Applications}, we recall the utility of designed pseudo-Laplacians by reproving meromorphic continuation of certain Eisenstein series and proving a spacing result for zeros of $\zeta_k(s)$ for $k$ a complex quadratic field with negative determinant, as in \cite{BG20} and \cite{Gar18}}. In Section \ref{Results}, we construct an automorphic Hamiltonian which has purely discrete spectrum on $L^2\left(SL_r(\Z)\backslash SL_r(\R)/SO(r,\R)\right)$, identify its ground state, and show how it can characterize a nuclear Fr\'echet automorphic Schwartz space.
\end{abstract}

\maketitle
\section{Introduction}
In 1972, Montgomery and Dyson heuristically related the spacing of zeros of the Riemann zeta function to the distribution of eigenvalues of random Hermitian matrices used in nuclear physics. The eigenstates of an important class of quantum mechanical operators satisfy this distribution, suggesting a mysterious connection between mathematical physics and number theory. {Some time in the 1910s}, P\'{o}lya and Hilbert independently speculated, in effect, that one might prove the Riemann Hypothesis if the zeros $s$ of $\zeta(s)$ corresponded to eigenvalues $\lambda_s = s(s-1)$ of a self-adjoint operator. While it is certainly possible to construct an operator with eigenvalues parametrized by zeros, {self-adjointness is apparently unverifiable by available
machinery \cite{BC95}}.

Efforts to prove that eigenvalues of certain self-adjoint operators correspond to zeros $s$ of $\zeta(s)$ have been dramatic. In his masters thesis  \cite{Ha77}, H. Haas attempted to numerically
compute eigenvalues $\lambda_s=s(s-1)$ of the invariant Laplace-Beltrami operator {(see Section \ref{Laplacian})} on $SL_2(\Z)\backslash\h$. H. Stark and D. Hejhal happened upon his list of $s$-values via A. Terras and recognized zeros of $\zeta(s)$ and the Dirichlet $L$-function $L(s, \chi_{-3})$, a striking discovery which suggested that the P\'olya-Hilbert dream might be realized with a natural operator. Hejhal immediately began efforts to replicate Haas’s results using more careful numerical methods but found that exactly
these zeros were missing. Naturally, he wondered how flawed numerical procedures could produce garbage output with such number theoretic significance. He determined that Haas had {misapplied Henrici's collocation method of numerically solving differential equations \cite{FHM67}}, inadvertently allowing non-smoothness of eigenfunctions at the corners $\omega = e^{2\pi i/3}$ of the usual fundamental domain for $SL_2(\Z)$, and had actually solved an inhomogeneous equation
$(\Delta-\lambda_s)u_s=\delta_\omega^{\text{afc}}$, with $\delta_\omega^{\text{afc}}=\sum_{\gamma\in\Gamma}\delta^{\mathfrak{H}}_\omega\circ\gamma$ the automorphic Dirac delta at the third and sixth roots of unity {\cite{Hej81}}. Since the “eigenfunctions” corresponding to zeros of $\zeta(s)$ were solutions to an inhomogeneous equation (rather than $(\Delta-\lambda_s)u=0)$, the special values in Haas’s list were not necessarily genuine eigenvalues of a self-adjoint operator, and no inferences (such as the Riemann
Hypothesis) could be made from the special $\lambda_s$'s.

Nonetheless, {certain properties of automorphic Green's functions} {(see Section \ref{GreenFn})} had already been investigated in \cite{Els73}, \cite{Neu73}, \cite{Fay77}, \cite{Nie73}, and others. For example, the constant term of $u_s$ is
\[
c_P u_s(iy)=\int_0^1 u_s(x+iy) dx=\frac{y^{1-s}E_s(\omega)}{2s-1},\hspace{5mm} \text{for }y \geq 1, 
\]
and it had long been known that 
\[
E_s(\omega)=\left(\frac{\sqrt{3}}{2}\right)^{s/2}\cdot\frac{\zeta(s)L(s,\chi_{-3})}{\zeta(2s)}. 
\]
Although these connections are provocative, there is no obvious connection to the (genuine) eigenvalues of a self-adjoint operator.

In the early 1980s, mathematical physicist Y. Colin de Verdi\`ere speculated on a way to promote these numbers to genuine eigenvalues by considering variants of $\Delta$ with additional true eigenfunctions \cite{CdV82}. In particular, he considered Friedrichs’ self-adjoint extension $\tilde\Delta_\delta$ of {a restriction of} $\Delta$ {(see Section \ref{Friedrichs})} to automorphic test functions in the kernel of $\delta_\omega^{\text{afc}}$. These {so-called} {pseudo-Laplacians} allow certain non-smooth functions as genuine eigenfunctions, so the special values in Haas’ list might be (true) eigenvalues of a self-adjoint operator. However, a technical hazard obstructs such a construction: the Friedrichs extension for second-order elliptic operators limits the non-smoothness of eigenfunctions to a $+1$-index Sobolev space, but solutions to $(\Delta-\lambda_s)u=\delta_\omega^{\text{afc}}$ are in a $1-\varepsilon$ Sobolev space. In his 1983 paper, Colin de Verdi\`ere suggested one might overcome this issue by projecting solutions to the orthogonal complement of the discrete spectrum of $L^2(SL_2(\Z)\backslash\h)$.

{Almost forty years later, \cite{BG20} validated and refined the observations of Hejhal and Colin de Verdi\`ere. After projecting, not exactly to the orthogonal complement of the discrete spectrum, but to the {larger} non-cuspidal spectrum, all eigenvalues $\lambda_s$, if any at all, of the pseudo-Laplacian $\tilde\Delta_\theta$ would correspond to zeros of $\zeta(s)$ and $L(s,\chi_{-3})$; however, assuming {Montgomery’s 1972 pair correlation conjecture} about the spacing of zeros, at most 94\% of the zeros of $\zeta(s)$ can appear in this basic construction. Since the alleged remaining 6\% probably aren’t special, it is reasonable to believe that no zeros appear in the simplest implementation of Colin de Verdi\`ere’s idea.}

Central, but not unique, to the work of Bombieri-Garrett is a sufficient analytical viewpoint to give genuine proofs of attractive heuristics, for example, from physics. After all, physicists \cite{Dir28}, \cite{Dir30}, \cite{Tho35}, \cite{BP35} used {solvable models} such as $(-\Delta - \lambda)u=\delta$ to successfully predict experimental outcomes since at least the 1930s, often rewriting the corresponding differential equation as a perturbation of $-\Delta$ by a {singular} potential $\delta$ (see Section \ref{Perturbations}). Apparently, these differential equations were only understood rigorously after \cite{BF61}.

The {spectral theory of unbounded self-adjoint operators} on global automorphic Sobolev spaces {(see Section \ref{Sobolev})} provides a reasonable context to discuss partial differential equations in automorphic forms, {from solvable models to eigenvalue equations of quantum mechanical Hamiltonians, perturbations of $-\Delta$ by smooth confining potentials $q$ (see Section \ref{Perturbations})}.  After establishing an appropriate analytical context in Section \ref{Background} and expounding on aforementioned results in Section \ref{Applications}, we construct an automorphic Hamiltonian $-\Delta+q$ whose self-adjoint Friedrichs extension has purely discrete spectrum on $SL_r(\Z)\backslash SL_r(\R)/SO(r,\R)$ and characterizes an automorphic Schwartz space as a nuclear Fr\'echet space.

\section{Background}
\label{Background}
\hspace{1em}{In this chapter, we recall classical results}. In Section \ref{Reduction}, we exhibit a single (adelic) Siegel set which covers automorphic quotients $G_k\backslash G_{\mathbb{A}}$ and define a suitable notion of height on such quotients to simplify computations. In Section \ref{Laplacian}, to facilitate construction of {an} automorphic Hamiltonian $-\Delta + q$, we recall a coordinate-free characterization of the invariant Laplacian $\Delta$ {on quotients $\Gamma\backslash G/K$}. In Sections \ref{Degenerate} and \ref{Epstein}, we recall analytical properties of degenerate Eisenstein series attached to certain maximal proper parabolic subgroups of $SL_r(\R)$, which motivate our choice of confining potential $q$. In Section \ref{Sobolev}, we characterize the global automorphic Sobolev spaces of functions {on which our Hamiltonian is defined}. In Section \ref{Schwartz}, we recall {Schwartz' kernel theorem} and offer a convenient notion of a nuclear Fr\'echet automorphic Schwartz space. In Section \ref{Friedrichs}, we recall Friedrichs' construction of self-adjoint extensions of densely-defined semi-bounded {symmetric} operators on Hilbert spaces.  In Section \ref{Perturbations}, we briefly recount the perturbation theory of linear operators, placing our Hamiltonians in a larger context. 

\subsection{Reduction Theory}
\label{Reduction}
{To simplify analysis in later sections, we exhibit a single (adelic) Siegel set that covers the quotient $SL_r(\Z)\backslash SL_r(\R)/SO_r(\R)$ and define a suitable notion of height. We largely follow the discussion in \cite{Gar18}*{Section 3}, which follows {Godement's Bourbaki talk on reduction theory \cite{God64}}, and refer the reader to the classic texts \cite{Bor66b} and \cite{Spr94}.  For convenience to the reader, we include some foundational material about the adele group $G_{\A}=GL_r(\A)$.}

\subsubsection{{The adele group $G_{\A}=GL_r(\A)$}}
\label{AdeleGroup}

Let $k$ be an arbitrary number field with integers $\mathfrak{o}$, completions $k_v$, and local rings of integers $\mathfrak{o}_v$ at non-archimedean places. Let $G_v=GL_r(k_v)$ with center $Z_v$. For non-archimedean $v$, let $K_v=GL_r(\mathfrak{o}_v)$. For real $v$, let $K_v$ be the standard orthogonal group $O_r(\R)=\{g\in GL_r(\R): g^\top g = 1_r\}$, and for complex $v$, let $K_v$ be the standard unitary group $U_r=\{g\in GL_r(\C): g^*g = 1_r\}$.

Let $\ell$ be the number of non-isomorphic archimedean completions of $k$. That is, $\ell=\ell_1+\ell_2$, where $\ell_1$ is the number of real completions, $\ell_2$ is the number of complex completions (not counting conjugates), and $[k:\Q]=\ell_1 + 2\ell_2$. Let $Z^+$ be the positive real scalar matrices diagonally imbedded across all archimedean places $v$ by the map 
\[
\delta: t\mapsto (...,t^{1/\ell},...)\hspace{7mm}(\text{for }t>0)
\]
The map $\delta$ gives a section of the idele norm map
\[
|\cdot|: t\mapsto \prod_v |t_v|_v.
\]
That is, $|\delta(t)|=t$.

The group $P_v$ of $v$-adic points of a standard parabolic $P=P^{F}$ is the stabilizer of $F_v\subset k_v^r$ in $G_v$ with the same shape as $P$. Similarly, $N_v = N^P_v$ and $M_v = M^P_v$ are the $v$-adic points of the unipotent radical $N^P$ and the standard Levi component $M^P$ of a standard parabolic $P$, respectively. As expected, we have Iwasawa decompositions $G_v = P_v\cdot K_v$ for a standard minimal parabolic $P$ and Cartan decompositions $G_v=K_vM_vK_v$ for $M$ the standard Levi component of the minimal parabolic (for proofs, see \cite{Gar18}*{Claims 3.2.1 and 3.2.2}, for example). The Cartan decompositions follow from the spectral theorem for symmetric or Hermitian operators at archimedean places and from the structure theorem for finitely-generated modules at finite places.

The corresponding adele group is $G_{\A}=GL_r(\A)$, $r$-by-$r$ matrices with entries in $\A$ and determinant in the ideles $\J$. This group is also a colimit of products
\[
G_S = \prod_{v\in S}G_v \times \prod_{v\notin S}K_v,
\]
where finite sets $S$ of places $v$ are ordered by containment. Similarly, $P_{\A}$, $M^P_{\A}$, $N^P_{\A}$, and $Z_{\A}$ are the adelic forms of those groups. We also let $K_{\A}=\prod_v K_v\subset G_{\A}$.

\subsubsection{{Reduction theory for $G_k\backslash G_{\A}$}}
\label{ReductionTheory}
For $k$ an arbitrary number field and $G_k= GL_r(k)$, we identify a suitable notion of height $\eta$ on $G_k\backslash G_{\A}$ and recall that a single adelic Siegel set $\mathfrak{S}_{t,C}$ covers the quotient $G_k\backslash G_{\A}$.

Let $P_k$, $M^P_k$, and $N^P_k$ be the corresponding subgroups of $G_k$ with entries in $k$. We omit a proof of discreteness of $G_k$ in $G_{\A}$, since the proof easily generalizes from the proof for $GL_2$ (see \cite{Gar18}*{Claim 2.2.1}, for example). Let 
\[
G^1= \{ g\in G_{\A}: |\det g|=1\}
\]
such that $G_{\A}=Z^+\times G^1$. By the product formula, $\prod_{v\leq \infty}|t|_v = 1$ for $t\in k^\times$, {we observe that $G_k\subset G^1$}, and, in particular, {$G_k$ is still discrete in $Z^+\backslash G_{\A}\approx G^1$}.

In contrast to the $GL_2$, the notion of a single numerical height for $GL_r$ is initially replaced by the family of $(r-1)$ standard positive simple roots
\[
\alpha_i \begin{pmatrix} m_1\\
& \ddots &
\\ & & m_r
\end{pmatrix} = \frac{m_i}{m_{i+1}}\hspace{7mm}(1\leq i < r),
\]
which are characters on the Levi component $M^{\text{min}}$ of the standard minimal parabolic $P^{\text{min}}=P^{1,...,1}$.  These roots make sense on $M_k$, $M_v$, and $M_{\A}$ and take values in $k^\times$, $k_v^\times$, and $\J$, respectively. 

For reasons that will become apparent, we need a notion of height on $\A^r$ to identify a suitable notion of height on $G_k\backslash G_{\A}$. Let $G_{\A}$ act on the right on $\A^r$ by matrix multiplication. For real places $v$ of $k$, the local height function on $k_v^r$ is $h_v(x) = \sqrt{x_1^2+...+x_r^2}$. For complex $v$, take $h_v(x)=|x_1|_{\C}+...+|x_r|_{\C}$, with $|z|_{\C}=|N_{\R}^{\C}z|_{\R}$ to avoid disturbing the product formula. For non-archimedean $v$, let $h_v(x)=\sup_{i}|x_i|_v$. {By design, the isometry groups of the height functions $h_v$ are the compact subgroups $K_v$. That is, the $K_v$ preserve heights $h_v$.}

{A primitive vector $x\in \A^r$ is of the form $x=x_o\cdot g$ for $g\in G_{\A}$ and $x_o\in k^r-\{0\}$}. For $x=(x_1,...,x_r)\in k^r$ and almost all non-archimedean places $v$, the components $x_i$ are in $\mathfrak{o}_v$ and have local greatest common divisor 1; furthermore, elements of the adele group $g\in G_{\A}$ are in $K_v$ almost everywhere, so this is not changed by multiplication by $g$.  Thus, primitive vectors $x=(x_1,...,x_r)$ have the property that the components $x_i$ are locally integrable and have local greatest common divisor 1 almost everywhere. Since almost all finite primes have local height 1, it makes sense to define a global height function on primitive vectors as the product of local height functions: $h(x)= \prod_v h_v(x_v)$. As recalled in \cite{Gar18}*{Claim 3.3.2}, for example, this global height function $h(x)$ on primitive vectors $x\in \A^r$ has several convenient properties which facilitate the reduction theory of $G_k\backslash G_{\A}$:
\begin{enumerate}
\item For fixed $g\in G_{\A}= GL_r(\A)$ and fixed $c>0$,
\[
\text{card}\left( k^\times \backslash \{x\in k^r - \{0\}: h(x\cdot g) < c\} \right) < \infty,
\]
\item For compact $C\subset G_{\A}$, there are positive implied constants such that for all $g\in C$ and primitive $x\in \A^r$,
\[
h(x)\ll_C h(x\cdot g)\ll_C h(x).
\]
\item For $t\in k^\times$, via the product formula $\prod_{v\leq \infty}|t|_v =1$ for $t\in k^\times$, 
\[
h(t\cdot x) = \prod_{v\leq \infty}h_v(t\cdot x)=\prod_{v\leq \infty}|t|_v\cdot h_v(x)=\prod_{v\leq \infty}|t|_v \cdot \prod_{v\leq \infty} h_v(x)=1\cdot h(x)
\]
\end{enumerate}

Finally, let 
\[
\eta_v(g_v)= |\det g_v|_v\cdot h_v(e_r\cdot g_v)^{-r}\hspace{10mm}(\text{for }g_v\in G_v, \{e_i\} \text{ the standard basis for }k^r)
\]
and $\eta(g)=\prod_v\eta_v(g_v)$ for $g=(g_v)_v\in G_{\A}$. In what follows, this is a suitable notion of height on $G_k\backslash G_{\A}$. We recall several analytic properties:

\begin{lemma}
The height function $\eta(g)= \prod_{v}|\det g_v|_v\cdot h_v(e_r\cdot g_v)^{-r}$ is right $K_{\A}$-invariant,  left $P_k^{r-1,1}$-invariant, and $Z_{\A}$-invariant.
\end{lemma}
\begin{proof}
The isometry groups of the height functions $h_v$ are the compact subgroups $K_v$, so $\eta(g)$ is right $K_{\A}$-invariant. By the product formula, for $p=\begin{pmatrix}a & b \\ 0 & d\end{pmatrix}\in P^{r-1,1}_k\subset G^1$,
\begin{align*} 
\eta(p\cdot g)&=|\det pg| \cdot h(e_r\cdot pg)^{-r}= |\det pg|\cdot h(de_r\cdot g)^{-r}\\
&= |\det p|\cdot |\det g|\cdot |d|^{-r}\cdot h(e_r\cdot g)^{-r} = 1\cdot |\det g|\cdot h(e_r\cdot g)^{-r} = \eta (g).
\end{align*}
For $z=\begin{pmatrix}t & &  & \\  &  & \ddots & \\  &  &  & t\end{pmatrix}\in Z_{\A}$
\begin{align*}
\eta(z\cdot g)&=|\det zg|\cdot h\left(e_r\cdot zg\right)^{-r}= |\det zg|\cdot h(te_r\cdot g)^{-r}\\
&= |\det z|\cdot |\det g|\cdot |t|^{-r}\cdot h(e_r\cdot g)^{-r} = |t|^r\cdot |t|^{-r}\cdot \eta(g) = \eta(g)
\end{align*}
\end{proof}
\begin{lemma}[For example, \cite{Gar18}*{Corollary 3.3.3}]
\label{heightLemma}
For any $g\in G_{\A}$, there are finitely-many $\gamma\in P_k^{r-1,1}\backslash G_k$ such that 
\[
\eta(\gamma\cdot g)>\eta(g).
\]
Thus, the supremum $\sup_{\gamma}\eta(\gamma\cdot g)$ is attained and is finite.
\end{lemma}
\begin{proof}
This lemma hinges on the {natural bijection $P^{r-1,1}_k\backslash G_k \longleftrightarrow k^\times \backslash (k^r - \{0\})$} given by
\[
P^{r-1,1}_k\cdot \begin{pmatrix} * & ...  & * \\
\vdots & \ddots & \vdots  \\
c_1 &... &c_{r} 
\end{pmatrix} \mapsto k^\times \cdot (c_1, ..., c_{r}).
\] 
for any invertible matrix with bottom row $(c_1, ..., c_{r})$. {Indeed, $G_k$ is transitive on non-zero vectors, and $P^{r-1,1}_k$ is the stabilizer, acting on the right, of the line (minus a point) $k^\times\cdot e_r$}. By property 1 of the global height function $h$, there are finitely-many $x\in k^\times \backslash (k^r - \{0\})$ such that $h(x\cdot g) < c$ (equivalently, such that $h(x\cdot g)^{-1} > c^{-1})$. Then, the bijection above and $G_k$-invariance of $|\det g|$ yields our assertion.
\end{proof}

Finally, we exhibit a single (adelic) Siegel set covering $G_k\backslash G_{\A}$. A standard Siegel set aligned to a (standard) parabolic $P$ is of the form
\[
\mathfrak{S}^P = \mathfrak{S}^P_{t,C} = \{g= nmk: n\in C, m\in M_{\A}, k\in K_{\A}, \text{ and }|\alpha_i(m)|\geq t, \text{ for }1\leq i < r\}
\]
for $t\in (0,\infty)$, $C\subset N^P_{\A}$ compact, and $|\cdot|$ the idele norm.  For $P=P^{min}$ the minimal parabolic, we have the following:

\begin{theorem}[For example, \cite{Gar18}*{Theorem 3.3.1}]
Given a number field $k$, there is some $t>0$ and compact $C\subset N^{P}_{\A}$ such that $G_k\cdot\mathfrak{S}^P_{t,C}=G_{\A}$.
\end{theorem}
\begin{proof}
{We include a proof to illustrate the role of the $P^{r-1,1}$ parabolic (see Section \ref{Parabolic}) in the reduction theory of $G_k\backslash G_{\mathbb{A}}$}. The proof proceeds by induction on $r$, reducing to the case $r=2$ treated in many sources, including \cite{Gar18}*{Section 2.2}.

Given $g\in G_{\A}$ and by Lemma \ref{heightLemma}, $h(x'\cdot g)$  is minimized at some $x\in k^{r}-\{0\}$ with $h(x\cdot g)>0$. Let $\gamma_o\in G_k$ such that $x=e_r\gamma_o$. Since $h(x\cdot g)=h(e_r\gamma_o\cdot g)$ is minimal, $\eta(\gamma_o g)$ is maximal among all values $\eta(\gamma\cdot \gamma_og)$ for $\gamma\in G_k$. By Iwasawa, there is $\theta\in K$ such that $q=\gamma_o g\theta\in P_{\A}^{r-1,1}$. {Then $h(\gamma_o g\theta)=|q_{r,r}|$ where $q_{i,j}$ is the $ij^{th}$ entry of $q$}, and $\eta(q)$ is maximal among all values $\eta(\gamma\cdot g)$ for $\gamma\in G_k$.

Let $H\subset M_{\A}^{r-1,1}$ be the subgroup of $G_{\A}$ fixing $e_r$ and stabilizing the subspace spanned by $e_1,...,e_{r-1}$, so $H\approx GL_{r-1}(\A)$. By induction on $r$, beginning at $r=2$ as treated in many sources including \cite{Gar18}*{Section 2.2}, by acting on $q=\lambda_o g\theta$ on the left by $H_k$ and on the right by $H_{\A}\cap K_{\A}$, we can suppose that $q\in P_{\A}^{min}$ and $|q_{i,i}/q_{i+1,i+1}|\geq t$ for $i<r-1$, without altering $\eta(q)$. Certainly, since 
\[
h(e_r\cdot q)\leq h(x'\cdot q)\hspace{5mm}\text{for all }x'\in k^r-\{0\},\]
we have $h(e_r\cdot q)\leq h(x'\cdot q)$ for the subset of vectors of the form $x'=(0,...,0,x_{r-1},x_r)$. Thus, the lower right 2-by-2 block $q'$ of $q$ is reduced as an element of $GL_2(\mathbb{A})$. This reduces to the $r=2$ case treated in \cite{Gar18}*{Theorem 2.2.7}, yielding 
\[
|q_{r-1,r-1}|/|q_{rr}|\geq t,\hspace{5mm}\text{for sufficiently small }t
\]
and proving the theorem.
\end{proof}

\subsection{The invariant Laplacian}
\label{Laplacian}
We characterize the invariant Laplacian descended from Casimir, prove its invariance on automorphic quotients $\Gamma\backslash G/K$, and recall facts about its eigenfunctions. To maintain some generality in our discussion and, incidentally, certify invariance, it is useful to characterize the invariant Laplacian on certain automorphic quotients without choosing coordinates.

Let $G$ be a semisimple Lie group with maximal compact subgroup $K$ and discrete subgroup $\Gamma$. The universal enveloping algebra $U\mathfrak{g}$ is a canonical quotient of the universal associative algebra $A\mathfrak{g}$, also known as the tensor algebra $\large\otimes^{\bullet}\mathfrak{g}$. The associative algebra $U\mathfrak{g}$ is universal in the sense that any linear map $\varphi: \mathfrak{g}\to B$ to an associative algebra $B$ which respects the Lie bracket
\[
\varphi([x,y])=\varphi(x)\varphi(y)-\varphi(y)\varphi(x)\hspace{5mm}\text{(for }x,y\in\mathfrak{g})
\]
induces a unique associative algebra homomorphism $\Phi:U\mathfrak{g}\to B$. That is, the following diagram commutes:
\[
\begin{tikzcd}
U\mathfrak{g} \arrow[dr, dashed,"\exists!\hspace{1mm}\Phi"] \\
\mathfrak{g}  \arrow[u,"i"] \arrow[r,"\forall\hspace{1mm}\varphi",swap] & B.
\end{tikzcd}
\]
{The} Casimir {element} $\Omega$ is the simplest non-trivial $G$-invariant element in $U\mathfrak{g}$ and gives a two-sided $G$-invariant differential operator on $G$ which descends to the invariant Laplacian on quotients $G/K$ {and $\Gamma\backslash G/K$}. Casimir can be characterized as the image of the $G$-invariant element $\text{id}_{\mathfrak{g}}\in \End_{\C}(\mathfrak{g})$ under the map $\zeta$ given by \\

\textbf{\underline{{Coordinate-free characterization of Casimir}}}
 \[
\begin{tikzcd}
\End_{\C}(\mathfrak{g}) \arrow[r, "\approx"] \arrow[d,dash] \arrow[bend left, "\zeta"]{rrrr} & \mathfrak{g}\otimes\mathfrak{g}^*  \arrow[r, "\approx\hspace{1mm}via\hspace{1mm}\langle\rangle"]  & \mathfrak{g}\otimes\mathfrak{g}  \arrow[r, "incl"] & A\mathfrak{g} \arrow[r, "quot"] & U\mathfrak{g} \arrow[d,dash] \\
\text{id}_{\mathfrak{g}}\arrow[rrrr, mapsto]   &  &  &  & \zeta(\text{id}_{\mathfrak{g}})=\Omega.
\end{tikzcd}
\]
The first map in the composition is the inverse of the canonical isomorphism $V\otimes V^*\to \End V$ for finite-dimensional vector spaces. The second map is induced from $\langle-,v\rangle\mapsto v: V^*\to V$ for $\langle-,-\rangle$ a $G$-equivariant, non-degenerate bilinear form on $\mathfrak{g}$, such as the trace form $\langle x,y\rangle=\text{tr}(xy)$ on $\mathfrak{so}_r, \mathfrak{sl}_r$, and $\mathfrak{gl}_r$. The third map is an inclusion, and the fourth is a quotient map.\\

{The above prescription tells us how to express the Casimir element $\Omega=\zeta(\text{id}_\mathfrak{g})$ in terms of any basis $x_1,...,x_r$ of $\mathfrak{g}$: For $\lambda_1,...,\lambda_r$ the corresponding dual basis of the dual $\mathfrak{g}^*$, characterized by $\lambda_i(x_j)=\delta_{ij}$, and $x_1^*,...,x_r^*$ the dual basis for $\mathfrak{g}$ in terms of $\langle,\rangle$, characterized by $\langle x_i, x_j^*\rangle=\delta_{ij}$,}

 \[
\begin{tikzcd}
\End_{\C}(\mathfrak{g}) \arrow[r, "\approx"] \arrow[d,dash] \arrow[bend left, "\zeta"]{rrrr} & \mathfrak{g}\otimes\mathfrak{g}^* \arrow[d,dash] \arrow[r, "\approx\hspace{1mm}via\hspace{1mm}\langle\rangle"]  & \mathfrak{g}\otimes\mathfrak{g} \arrow[d,dash]\arrow[r, "incl"] & A\mathfrak{g} \arrow[d,dash] \arrow[r, "quot"] & U\mathfrak{g} \arrow[d,dash] \\
\text{id}_{\mathfrak{g}}\arrow[r, mapsto]   & \sum_{i} x_i\otimes\lambda_i \arrow[r, mapsto]  & \sum_i x_i\otimes x_i^* \arrow[r, mapsto] & \sum_i x_i\otimes x_i^* \arrow[r, mapsto] & \sum_i x_ix_i^*=\Omega.
\end{tikzcd}
\]
\\
That is, we have an intrinsic description of the Casimir element which does not depend on a choice of basis $x_1,...,x_r$. Furthermore, we can show that $\Omega=\zeta(\text{id}_{\mathfrak{g}})$ is $G$-invariant in $U\mathfrak{g}$ without showing by change-of-basis that the defined object is independent of choice of basis:
{\begin{theorem}
\label{CasInv}
The Casimir operator $\Omega = \zeta(\text{id}_{\mathfrak{g}})$ is a $G$-invariant element of $U\mathfrak{g}$.
\end{theorem}
\begin{proof}
Since the action of $G$ respects all maps in the composition, $\zeta$ is $G$-equivariant. Furthermore, $\text{id}_\mathfrak{g}$ commutes with any endomorphism of $\mathfrak{g}$, so $\Omega=\zeta(\text{id}_\mathfrak{g})$ is certainly $G$-invariant element of $U\mathfrak{g}$:
\[
g\hspace{1mm}\zeta(\text{id}_\mathfrak{g})\hspace{1mm}g^{-1} = \zeta(g\hspace{1mm}\text{id}_\mathfrak{g}\hspace{1mm}g^{-1})= \zeta(g\hspace{1mm}g^{-1}\hspace{1mm}\text{id}_\mathfrak{g})=\zeta(\text{id}_\mathfrak{g}).
\]
\end{proof}}
{Incidentally, this intrinsic description renders trivial the proof of invariance of Casimir on automorphic quotients:} 

{\begin{theorem}
For $G$ a semisimple Lie group with maximal compact $K$ and discrete subgroup $\Gamma$, the Casimir operator $\Omega = \zeta(\text{id}_{\mathfrak{g}})$ descends to an invariant differential operator on  $\Gamma\backslash G/K$.
\end{theorem}
\begin{proof}
On one hand, the ({right} differentiation) action of $x\in\mathfrak{g}$ on functions $f$ on $G$ by
\[
(x\cdot f)(g)=\frac{d}{dt}\bigg{|}_{t=0}f(g\cdot e^{tx})
\]
is {left} $G$-invariant (and therefore left $\Gamma$-invariant) by associativity. Certainly, this extends to the left $G$-invariant action of $U\mathfrak{g}$. On the other hand, if the differential operator comes from the $\text{Ad } G$-invariant elements of $U\mathfrak{g}$ given by
\[
(U\mathfrak{g})^G=\{\alpha\in U\mathfrak{g}: g\alpha g^{-1}=\alpha, \text{ for all }g\in G\}, 
\] 
then the action is right $G$-invariant (and therefore right $K$-invariant) by the commutation property. By Theorem \ref{CasInv}, $\Omega\in (U\mathfrak{g})^G$ and, thus, descends to the quotient $\Gamma\backslash G/K$.
\end{proof}}
{\subsection{Eigenvalues of integral operators on eigenfunctions of Casimir}}
As in the proof of a pre-trace formula in, for example, \cite{Gar18}*{Section 12.1}, we recall that the eigenvalues of integral operators on Eisenstein series or strong-sense cuspforms are essentially determined by their eigenvalues for the Laplacian.  This follows from the non-trivial fact that a suitable representation $V_f$ generated by right translates of a $\Delta$-eigenfunction $f$ is isomorphic to a subquotient of the principal series representation $I_s$.   That subtler point is an instance of a subquotient theorem \cite{HC54}, which was eventually strengthened to the subrepresentation theorem \cite{Cas78/80},\cite{CM82}. Eigenvalues depend only on the isomorphism class of $V_f$, so they can be computed on the (much simpler) subquotient of the principal series representation. 

We follow the discussion in \cite{Gar18}*{Chapter 12}. Consider an integral operator attached to a compactly supported measure $\eta$ on the group $G$ acting on a quasi-complete, locally convex topological vector space $V$ of automorphic forms. For a continuous action $G\times V\to V$, the action of $\eta$ is
\[
\eta \cdot v = \int_{G}g\cdot v\hspace{1mm} d\eta(g),\hspace{5mm}\text{for }v\in V,
\]
as a Gelfand-Pettis integral.

Quasi-complete, locally convex topological vector spaces $V,W$ with continuous actions $G\times V\to V$ and $G\times W\to W$ are representations of $G$, and continuous $\C$-linear maps $T:V\to W$ respecting the action of $G$ are $G$-homomorphisms. By properties of Gelfand-Pettis integrals, $G$-homomorphisms commute with the action of integral operators attached to compactly supported measures:
\[
T(\eta \cdot v) = T\left(\int_{G}g\cdot v\hspace{1mm} d\eta(g)\right)= \int_{G}T(g\cdot v)\hspace{1mm} d\eta(g)=\int_{G}g\cdot T(v)\hspace{1mm} d\eta(g) = \eta\cdot T(v).
\]
In particular, $G$-homomorphisms preserve the eigenvalues and eigenvectors of integral operators. 

For $v$ in a (quasi-complete, locally-convex) $G$-representation space $V$, the subrepresentation generated by $v$ is the topological closure of the span of finite linear combinations of images $g\cdot v$ of $v$ by $g\in G$. By properties of Gelfand-Pettis integrals (as shown in \cite{Gar18}*{Claim 12.1.3}, for example), a strong-sense cuspform or Eisenstein series $f$ is the unique $K$-invariant vector in the subrepresentation $V_f$ it generates under right translation, up to a constant. For any left and right $K$-invariant compactly supported measure $\eta$, the integral operator action 
\[
(\eta \cdot f)(x) = \int_{G}g\cdot f(x)\hspace{1mm} d\eta(g) = \int_{G} f(xg)\hspace{1mm} d\eta(g)
\]
produces another right $K$-invariant vector in the subrepresentation space $V_f$; thus $\eta\cdot f$ is a scalar multiple of $f$. Let $\chi_f(\eta)\in \C$ denote the eigenvalue:
\[
\eta\cdot f = \chi_f(\eta) \cdot f.
\]
Since $G$-homomorphisms preserve the eigenvalues of integral operators, the scalar $\chi_f(\eta)$ can be computed in any image of $V_f$. For $\Omega\cdot f = r\ell(r-\ell)s(s-1)\cdot f$, the representation generated by $f$ is isomorphic to an unramified principal series
\[
I_s = \left\{\varphi \in C^\infty(G): \varphi\left(\begin{pmatrix}a & *\\ 0 & d\end{pmatrix}\cdot g\right)=\frac{|\det a|^{(r-\ell)s}}{|\det d|^{\ell s}}\cdot \varphi(g)\right\},\hspace{5mm}(\text{with }s\in \C)
\]
under right translation by $G$. The Iwasawa decomposition $G = P^{\ell, r-\ell}\cdot K$ shows that the space of $K$-fixed vectors is one-dimensional. {By an explicit form of the proof of the Harish-Chandra isomorphism (see \cite{HC51} and \cite{Gar17b}, for example)}, we can compute eigenvalues of $\Omega$ on Eisenstein series or strong-sense level-one/spherical cuspforms by computing eigenvalues on $I_s$.

\subsection{Degenerate Eisenstein series attached to maximal proper parabolics} 
In this section, we define and recall analytical properties of degenerate Eisenstein series $E_{s,\varphi}^P$ attached to maximal proper parabolic subgroups {$P=P^{\ell,r-\ell}\subset SL_r(\R)$}. For this classical geometric algebra, we primarily follow the discussion in \cite{Gar18}*{Sections 3.1, 3.11}.  

\subsubsection*{Parabolic subgroups of $SL_r(\R)$}
\label{Parabolic}
Let $k$ be a number field. A flag $F$ in $k^r$ is a nested sequence of $k$-subspaces
\[
V_1\subset ... \subset V_m \subset k^r
\]
The corresponding parabolic subgroup $P = P^F$ is the stabilizer of the flag $F$. The maximal proper parabolic subgroups are stabilizers $P^{V}$ of flags consisting of single proper subspaces $V\subset k^r$. 

With $e_1,e_2,..., e_r$ the standard basis for $k^r$, identify $k^d=ke_1+...+k e_d$. By transitivity of $G$ on ordered bases of $k^r$, every orbit in the action of $G$ on flags has a unique representative among the standard flags, namely for some ordered partition $d_1+d_2+...+d_m=r$ with $0<d_j\in  \Z$, the corresponding standard flag is
\[
F^{d_1,...,d_m} = \left(k^{d_1}\subset k^{d_1+d_2}\subset ....\subset k^{d_1+d_2+...+d_m}\right).
\]

The standard proper parabolic subgroup $P^{d_1,...,d_m}$ of $G$ is the stabilizer of the flag $F^{d_1,...,d_m}$. The standard {maximal} proper parabolics are the block-upper triangular matrices
\[
P^{\ell, r-\ell}=\left\{\begin{pmatrix} a & b\\
0 & d
\end{pmatrix}: a \in SL_\ell(k), b = \ell \times (r-\ell), d\in SL_{r-\ell}(k)\right\},
\]
where the off-diagonal blocks are sized to fit. The unique {associate} parabolic of $P=P^{\ell,r-\ell}$ is the subgroup $P^{r-\ell,\ell}$, defined as expected. We observe that $P^{\ell,r-\ell}$ is {self-associate} when $r=2\ell$.

The unipotent radical $N^P$ of a parabolic $P$ stabilizing a flag $F=(V_1\subset ... \subset V_m)$ is the subgroup that fixes the quotients $V_m / V_{m-1}$ pointwise. For the standard maximal proper parabolic $P^{\ell, r-\ell}$, the unipotent radical is
\[
N^{\ell, r-\ell} = \left\{\begin{pmatrix} 1_\ell &*\\
0 & 1_{r-\ell}
\end{pmatrix}
\right\}.
\]

The standard Levi-Malcev component $M^P$ is the subgroup of $P$ with all the blocks above the diagonal 0. For the standard maximal proper parabolic $P^{\ell, r-\ell}$, the Levi component is
\[
M^{\ell, r-\ell} = \left\{\begin{pmatrix} a & 0\\
0 & d
\end{pmatrix}
\right\}.
\]

{The proof of the Iwasawa decomposition $G = P^{\ell, r-\ell} \cdot K$ is an iteration of the $r=2$ case, and the Levi-Malcev decomposition $P^{\ell, r-\ell} = N^{\ell, r-\ell} \cdot M^{\ell, r-\ell}$ is simply an expression of the behavior of matrix multiplication in block decompositions.}
\subsubsection*{Degenerate Eisenstein series}
\label{Degenerate}
{Degenerate Eisenstein series play a subtler role in Plancherel or the spectral decomposition of $L^2(\Gamma\backslash G/K)$ than do cuspidal-data Eisenstein series; however, as was known to Selberg and appeared in \cite{Lan76}, they do appear as $(r-2)$-fold residues of a minimal-parabolic Eisenstein series. However, as multi-residues of Eisenstein series are not necessary for the current discussion, we adapt \cite{Gar18}*{Section 3.11} and \cite{Gar09}.} 

Let $G=SL_r(\R)$, $P= P^{\ell,r-\ell}$ a maximal proper parabolic subgroup, $K=SO(r,\R)$ the maximal compact, $\Gamma = SL_r(\Z)$, and
\[
\varphi^P_{s}(nmk)=\frac{|\det a|^{(r-\ell)s}}{|\det d|^{\ell s}},\hspace{5mm} m=\begin{pmatrix} a & 0\\ 0 & d\end{pmatrix}.
\]
Note that $\varphi_s^P$ is the unique spherical vector in a (degenerate) unramified principal series representation of $G$, and is, therefore, left $N_{\A}^P$ invariant. Choosing a convenient basis for $\mathfrak{g}$, we can compute the $\Omega$-eigenvalue explicitly:
\begin{lemma}
\label{Eigenvalue}
For $\Omega$ the Casimir operator on $G$,
\[
\Omega\cdot \varphi^P_s=r\ell(r-\ell)(s^2-s)\cdot\varphi^P_{s}.
\]
\end{lemma}
\begin{proof}
Let $h_i$ be the diagonal matrix with 1 at the $i^{th}$ position and 0's elsewhere; furthermore, for $i<j$, let $x_{ij}$ be the matrix with a 1 at the $ij^{th}$ location and 0's elsewhere, and for $i>j$, let $y_{ij}$ be the matrix with a 1 at the $ij^{th}$ location and 0's elsewhere. The $h_i, x_{ij}$, and $y_{ij}$, form an orthogonal basis $\{x_k\}$ for $\mathfrak{g}$ with $h_i^*=h_i$, $x_{ij}^*=y_{ji}$, $y_{ij}^*=x_{ji}$, and $[x_{ij},y_{ji}]=-h_i + h_j$. Thus, we have
\begin{align*}
\Omega &= \sum_{k} x_kx^*_k= \sum_{1\leq i\leq r} h_i^2 + \sum_{i<j}x_{ij}y_{ji}  + \sum_{i>j}y_{ij}x_{ji} = \sum_{1\leq i\leq r} h_i^2 + \sum_{i<j}x_{ij}y_{ji}  + y_{ji}x_{ij} \\
&= \sum_{1\leq i\leq r}h_i^2 + \sum_{i<j}2x_{ij}y_{ji}  + [x_{ij}, y_{ji}]=\sum_{1\leq i\leq r}h_i^2 + \sum_{i<j}2x_{ij}y_{ji}  - h_i + h_j
\end{align*}
Since $x_{ij}\in \mathfrak{n}$ and acts first (after conjugating) in the middle term $\sum_{i<j}2x_{ij}y_{ij}$,  this term acts by 0 on left $N^P$ invariant functions, such as $\varphi^P_s$. Thus,
\[
\Omega \cdot \varphi^P_s = \sum_{i}h_i^2  \cdot \varphi^P_s + \sum_{i<j} (-h_i+h_j) \cdot \varphi^P_s.
\]
Although, as usual, $h\in \mathfrak{g}$ acts on the \textit{right}, Casimir commutes with conjugation by $G$, so we can let the associated differential operators $h_i$ act on the \textit{left}. Thus, by direct calculation,
\begin{align*}
h_i\cdot \varphi^P_s(g) &= \begin{cases} s(r-\ell)\cdot \varphi_s^P(g) & 1 \leq  i \leq\ell \\
-s\ell\cdot \varphi^P_s(g) & \ell < i \leq r,
\end{cases}
\end{align*}
such that,
\begin{align*}
 \sum_{i}h_i^2  \cdot \varphi^P_s &=  \sum_{1\leq i\leq \ell }h_i^2  \cdot \varphi^P_s  +  \sum_{\ell<i\leq r}h_i^2  \cdot \varphi^P_s \\
 &= [\ell\left(s(r-\ell)\right)^2 +(r-\ell)(-s\ell)^2] \cdot \varphi_s^P = s^2\ell(r-\ell)r\cdot\varphi_s^P\\
 \sum_{i<j}(-h_i+h_j)\cdot \varphi^P_s &= \sum_{1 \leq i<j\leq \ell}0 \cdot \varphi^P_s + \sum_{\ell< i<j\leq r}0 \cdot \varphi^P_s +  \sum_{1\leq i\leq \ell < j\leq r}(-h_i+h_j)\cdot \varphi^P_s \\
 &= -(r-\ell)\sum_{1\leq i\leq \ell}h_i\cdot \varphi^P_s + \ell \sum_{\ell < j \leq r} h_j\cdot \varphi^P_s \\
 &= [-(r-\ell)\ell s(r-\ell) + \ell(r-\ell)(-s\ell)]\cdot \varphi_s^P = -s\ell(r-\ell)r\cdot \varphi_s^P.
\end{align*}
Summing these identities yields our claim.
\end{proof}

The degenerate Eisenstein series associated to $P$ and corresponding to $\varphi^P_s$ is 
\[
E^P_{s,\varphi}(g) = \sum_{{\gamma\in P\cap \Gamma \backslash \Gamma}} \varphi^P_{s}(\gamma\cdot g).
\]
We recall basic analytic properties which will be useful in the sequel:

\begin{theorem}
The degenerate Eisenstein series $E^P_{s,\varphi}(g)$ converges absolutely and uniformly on compacta for $\text{Re}(s)>{1}$.
\end{theorem}
\begin{proof}
For example, this follows from Godement's Criterion (see \cite{Bor66a}), as in the proof of \cite{Gar18}*{Claim 3.11.1}.
\end{proof}

{Of course, since $\varphi^P_s$ is right $K$-invariant and winding up to form the corresponding Eisensteins series is a $G$-homomorphism}, the degenerate Eisenstein series is also right $K$-invariant in the region of convergence. Furthermore, the meromorphic continuation of the degenerate Eisenstein series is right $K$-invariant by the identity principle.

{\begin{theorem}
The degenerate Eisenstein series $\displaystyle E^P_{s,\varphi}$ is an eigenvector of Casimir $\Omega$ with eigenvalue $\lambda_s = r\ell(r-\ell)s(s-1)$. In particular, the eigenvalue is invariant under $s\to 1-s$.
\end{theorem}
\begin{proof}
Granting convergence in the $C^{\infty}$ topology, in $\text{Re}(s)>1$, and using the fact that $\Omega$ commutes with translations by $\Gamma$, letting $\Omega\cdot \varphi^P_{s} = \lambda_s\cdot \varphi^P_{s}$ (as in Lemma \ref{Eigenvalue}),
\begin{align*}
\Omega \cdot E^P_{s,\varphi} &= \Omega \sum_{{\gamma\in P\cap \Gamma \backslash \Gamma}} \varphi^P_{s}(\gamma\cdot g) = \sum_{{\gamma\in P\cap \Gamma \backslash \Gamma}} (\Omega \cdot \varphi^P_{s})(\gamma\cdot g) \\
&= \lambda_s \sum_{{\gamma\in P\cap \Gamma \backslash \Gamma}} \varphi^P_{s}(\gamma\cdot g)=\lambda_s\cdot E^P_{s,\varphi}.
\end{align*}
Finally, the eigenvalue property holds for the meromorphic continuation of $E^P_{s,\varphi}$ by the identity principle.
\end{proof}
}

\subsection{Epstein zeta functions}
\label{Epstein}
For $Q$ a real, positive-definite $r\times r$ matrix, the corresponding Epstein zeta function is
\[
Z_r(Q, s) = \sum_{0\neq v\in \Z^r}(v\hspace{1mm}Q\hspace{1mm}v^\top )^{-s}, \hspace{5mm}\text{Re}(s)>\frac{r}{2}.
\]
\cite{Eps03} introduced this function as a more general function satisfying a functional equation similar to the one satisfied by the Riemann zeta function. With $Q=gg^\top $ for $g\in SL_r(\R)$, we obtain a family of functions on $SL_r(\R)$, parametrized by $s\in \C$. To clarify the normalization (such that the first pole is at $s=1$) and make certain analytical assertions precise, recall:

\begin{lemma} For $g\in SL_r(\R)$, the Epstein zeta function $Z_r(gg^\top ,s)$ is essentially a degenerate Eisenstein series. More precisely, 
\[
Z_r\left(gg^\top ,\frac{rs}{2}\right)=2\zeta\left({rs}\right) E^P_{s,\varphi}(g),
\] 
for $\zeta(s)$ Riemann's zeta function and $E^P_{s,\varphi}(g)$ a degenerate Eisenstein series associated to the $(r-1,1)$-parabolic subgroup of $SL_r(\R)$.
\label{EpEis}
\end{lemma}
\begin{proof}
{We derive a re-normalization of \cite{Gol15}*{Equation 10.7.4}}. Since $v(gg^\top )v^\top  = |v\cdot g|^2$ is homogeneous of degree 2, we can rewrite this Epstein zeta function as
\[
Z_r\left(gg^\top ,\frac{rs}{2}\right) = \sum_{0\neq v\in \Z^r}{|v\cdot g|^{-2(rs/2)}} = \zeta(rs) \sum_{\text{ prim. }v\in \Z^r}{|v\cdot g|^{-rs}},
\]
where $\zeta(rs)$ appears by pulling out the gcd's of $v$.

Let $\Gamma = SL_r(\Z)$,  $P=P^{r-1,1}$ the maximal proper parabolic subgroup 
\[
P=\left\{\begin{pmatrix}(r-1)\times (r-1) & * \\ 0 & 1\times 1 \end{pmatrix} \right\}
\]
of $SL_r(\R)$, and $\Lambda$ the set of primitive vectors in $\Z^r$. Notice $\Gamma$ acts transitively on $\Lambda$. By the orbit-stabilizer theorem, $\{\pm 1\}\backslash\Lambda$ is in bijection with the coset space $\Gamma_x\backslash\Gamma$,  for $\Gamma_x$ the isotropy group of the $\Z$-module generated by $x\in \Lambda$. Since $P\cap \Gamma$ stabilizes the lattice generated by the standard basis vector $e_r=(0,...,0,1)\in \Lambda$, $P\cap \Gamma$ is the isotropy group $\Gamma_{e_r}$. Thus, $\{\pm 1\}\backslash\Lambda$ is in bijection with the coset space $P\cap\Gamma\backslash\Gamma$ by 
\[
 \Psi: \{\pm 1\} \cdot (c_1, ..., c_{r}) \mapsto
(P\cap \Gamma) \cdot \begin{pmatrix} * & ...  & * \\
\vdots & \ddots & \vdots  \\
c_1 &... &c_{r}
\end{pmatrix}.
\] 

Let $\varphi_s^P$ be the data for the degenerate Eisenstein series as in \ref{Degenerate}. By the Iwasawa decomposition $G=N^P\cdot M^P\cdot K$, for any $g=\begin{pmatrix} 1_{r-1} &*\\ 0 & 1\end{pmatrix}\begin{pmatrix} a & 0\\ 0 & d\end{pmatrix} \cdot k\in SL_r(\R)$,
\begin{align*}
\varphi^P_s(g) = \varphi^P_s(nmk)=|d|^{-rs} = |e_r\cdot g|^{-rs}.
\end{align*}
Similarly, for $v\in \{\pm 1\}\backslash \Lambda$, 
\[
\varphi^P_s\left(\Psi(v) g\right)= |e_r\cdot\Psi(v) g|^{-rs}=|v\cdot g|^{-rs}.
\]
Thus, we conclude that
\begin{align*}
Z_r\left(gg^\top ,\frac{rs}{2}\right) &= \zeta(rs) \sum_{v\in \Lambda}{|v\cdot g|^{-rs}} = 2\zeta(rs) \sum_{v\in \{\pm 1\}\backslash\Lambda}{|v\cdot g|^{-rs}} \\
&= 2\zeta(rs)\sum_{\Psi(v)\in P\cap \Gamma\backslash \Gamma} \varphi_s^P(\Psi(v) g) = 2\zeta(rs)E^P_{s,\varphi}(g),
\end{align*}
yielding our claim.
\end{proof}

{We recall the proof of meromorphic continuation and functional equation, as in many sources such as \cite{Gol15}*{Section 10.7}, of the degenerate Eisenstein series $E^P_{s,\varphi}$ by Poisson summation \`a la Riemann's proof for $\zeta(s)$ (see \cite{Rie59} and \cite{Gar15}). Certainly, this proof yields meromorphic continuation and functional equation of the Epstein zeta function $Z_r(gg^{\top},s)$, as in \cite{Eps03} and \cite{Ter73}}.

\label{EpsteinMero}
\begin{theorem} The degenerate Eisenstein series $E^P_{s,\varphi}$ associated to the $(r-1,1)$-parabolic subgroup of $SL_r(\R)$ meromorphically continues to $\C$ with a simple pole at {$\displaystyle s=1$ with residue = $\displaystyle\frac{\pi^{r/2}}{{2}\Gamma\left(r/2\right)\zeta(r)}$.} Furthermore, $E^P_{s,\varphi}$ satisfies the functional equation
\[
\pi^{-\frac{rs}{2}}\Gamma\left(\frac{rs}{2}\right)\zeta(rs) E^P_{s,\varphi}(g)=\pi^{-\frac{r}{2}(1-s)}\Gamma\left(\frac{r(1-s)}{2}\right)\zeta\left(r-rs\right) {E^P_{1-s,\varphi}\left((g^\top)^{-1}\right)}.
\]
\end{theorem}

\begin{proof} As usual, let $v\in \mathbb{Z}^r$ act on $SL_r(\mathbb{R})$ by matrix multiplication, and consider the Gaussian
\[
\varphi(v)=e^{-\pi|v|^2}.
\]
For $g\in SL_r(\mathbb{R})$, let
\[
\Theta(g)=\sum_{v\in \mathbb{Z}^r}\varphi(v\cdot g)=\sum_{v\in \mathbb{Z}^r}e^{-\pi|v\cdot g|^2}
\]
and (for reasons which will become apparent) consider the integral representation
\[
\mathcal{I}(g)=\int_0^\infty t^{2s}(\Theta(tg)-1)\frac{dt}{t}=\int_0^\infty t^{2s}\left(\sum_{0\neq v\in \Z^r}e^{-\pi|v\cdot tg|^2}\right)\frac{dt}{t}
\]
where $t$ acts on $SL_r(\mathbb{R})$ by scalar multiplication. Since $\Theta$ decays rapidly, we can integrate $\mathcal{I}$ termwise. By the change of variables $t\mapsto t/(\sqrt{\pi}|v\cdot g|)$ and Lemma \ref{EpEis},
\begin{align*}
    \mathcal{I}(g)&=\sum_{0\neq v\in \mathbb{Z}^r}\int_0^\infty t^{2s} e^{-\pi|v\cdot tg|^2}\frac{dt}{t}=\sum_{0\neq v\in \mathbb{Z}^r}(\sqrt{\pi}|v\cdot g|)^{-2s}\int_0^\infty t^{2s} e^{-t^2}\frac{dt}{t}\\
    &=\frac{1}{2}\pi^{-s}\sum_{0\neq v\in \mathbb{Z}^r}|v\cdot g|^{-2s}\int_0^\infty u^{s} e^{-u}\frac{du}{u}=\frac{1}{2}\pi^{-s}\Gamma(s)Z_r\left(gg^{\top},s\right )\\
    &= \pi^{-s}\Gamma(s)\zeta(2s) E^P_{2s/r,\varphi}(g).
\end{align*}
\indent We will use this integral representation to prove the meromorphic continuation of $E^P_{s,\varphi}$ as in Riemann's corresponding argument for $\zeta(s)$. First, break the integral into two parts at $t=1$. By elementary estimates, the integral from 1 to $\infty$ is uniformly and absolutely convergent. Thus
\[
\int_1^\infty t^{2s}(\Theta(tg)-1)\frac{dt}{t}=\text{entire in }s.
\]
To prove meromorphy of the integral from 0 to 1, we use Poisson summation on the kernel of $\mathcal{I}$. By direct calculation, 
\[
{\text{Fourier transform of }(v\to \varphi(v\cdot tg))= v\to {t^{-r}}\text{det}(g)^{-1} \varphi\left(v\cdot \frac{(g^\top )^{-1}}{t}\right),}
\]
such that Poisson summation asserts
\[
{\Theta(tg)= {t^{-r}}\text{det}(g)^{-1} \Theta\left(\frac{(g^\top)^{-1}}{t}\right)}.
\]
For now, assume {Re($s)>\frac{r}{2}$}. By the Poisson summation assertion, change of variables $t\mapsto 1/t$, and explicit evaluation of some elementary integrals, 
\begin{align*}
    \int_0^1 t^{2s}(\Theta(tg)-1)\frac{dt}{t}&=\int_0^1 t^{2s}\left({t^{-r}}\text{det}(g)^{-1}[\Theta\left(\frac{(g^\top)^{-1}}{t}\right)-1]+{t^{-r}}\text{det}(g)^{-1}-1\right)\frac{dt}{t}\\
    &=\int_1^\infty t^{-2s}\left({t^{r}}\text{det}(g)^{-1}[\Theta\left(t(g^\top)^{-1}\right)-1]+{t^{r}}\text{det}(g)^{-1}-1\right)\frac{dt}{t}\\
    &=\text{det}(g)^{-1}\int_1^\infty {t^{r-2s}}\left(\Theta\left(t(g^\top)^{-1}\right)-1\right)\frac{dt}{t}+{\frac{\text{det}(g)^{-1}}{2s-r}}-\frac{1}{2s}\\
    &= \text{entire in }s+ {\frac{\text{det}(g)^{-1}}{2s-r}}-\frac{1}{2s},
\end{align*}
since the integral from 1 to $\infty$ converges nicely by elementary estimates. Thus, for $g\in SL_r(\R)$,
\begin{align*}
\pi^{-s}\Gamma(s)\zeta(2s) E^P_{2s/r,\varphi}(g)&=\int_1^\infty t^{2s}(\Theta(tg)-1)\frac{dt}{t} + \int_0^1 t^{2s}(\Theta(tg)-1)\frac{dt}{t}\\
&=\int_1^\infty t^{2s}(\Theta(tg)-1)\frac{dt}{t}+\int_1^\infty {t^{r-2s}}\left(\Theta\left(t(g^\top)^{-1}\right)-1\right)\frac{dt}{t}\\
&\hspace{10mm}+{\frac{1}{2s-r}}-\frac{1}{2s}\\
&=\text{entire in }s + {\frac{1}{2s-r}}-\frac{1}{2s}.
\end{align*}
The rational expressions have meromorphic continuations, so the right-hand side of the expression yields a meromorphic continuation of $E^P_{s,\varphi}$. {The right-hand side of the expression is also visibly invariant under {$s\to (\frac{r}{2}-s)$}}, so $E^P_{2s/r,\varphi}$ satisfies the functional equation
\[
\pi^{-s}\Gamma(s)\zeta(2s) E^P_{2s/r,\varphi}(g)=\pi^{-(\frac{r}{2}-s)}\Gamma\left(\frac{r}{2}-s\right)\zeta(r-2s) {E^P_{1-2s/r,\varphi}\left((g^\top)^{-1}\right)}.
\]
Renormalizing the complex parameter in the degenerate Eisenstein series,
\[
\pi^{-\frac{rs}{2}}\Gamma\left(\frac{rs}{2}\right)\zeta(rs) E^P_{s,\varphi}(g)=\pi^{-\frac{r}{2}(1-s)}\Gamma\left(\frac{r(1-s)}{2}\right)\zeta\left(r-rs\right) {E^P_{1-s,\varphi}\left((g^\top)^{-1}\right)}.
\]

Finally, we see that the only poles of $\pi^{-s}\Gamma(s)\zeta(2s) E^P_{2s/r,\varphi}$ are at {$\displaystyle s=0$ and $r/2$} with constant residues $-\frac{1}{2},\frac{1}{2}$. Using standard facts about $\Gamma(s)$ and $\zeta(s)$, we recover assertions for $E^P_{s,\varphi}$ itself. At $s=0$, $\Gamma(s)$ has a simple pole of residue $1$ and (the analytic continuation of) $\zeta(s)$ takes value $-\frac{1}{2}$, so $E^P_{2s/r,\varphi}$ and $E^P_{s,\varphi}$ are holomorphic at $s=0$. At $s=r/2$, $\pi^{-s}\Gamma(s)$ is holomorphic. Since the infinite product for $\zeta(2s)$ converges for $\text{Re}(s)>1/2$, $\zeta(2s)$ is nonzero in this region. {Thus $E^P_{2s/r,\varphi}$ has a simple pole at $s=r/2$ with constant residue $\displaystyle\frac{\pi^{r/2}}{2\Gamma\left(\frac{r}{2}\right)\zeta(r)}$} and poles at {$\rho/2$ for all non-trivial zeros $\rho$ of $\zeta(s)$ in $0<\text{Re}(s)<\frac{1}{2}$}.  Renormalizing the complex parameter in the degenerate Eisenstein series, $E^P_{s,\varphi}$ has a {simple pole at $s=1$ with constant residue $\displaystyle\frac{\pi^{r/2}}{2\Gamma\left(\frac{r}{2}\right)\zeta(r)}$ and poles at $\rho/r$ for all non-trivial zeros $\rho$ of $\zeta(s)$ in $0<\text{Re}(s)<\frac{1}{2}$.}

\end{proof}
By Lemma \ref{EpEis},
\begin{corollary}
For $g\in SL_r(\R)$, the Epstein zeta function $Z_r(gg^\top,s)$ meromorphically continues to $\mathbb{C}$ with a simple pole at {$s=r/2$ with constant residue $\displaystyle\frac{\pi^{r/2}}{\Gamma\left(r/2\right)}$}. Furthermore, $Z_r(gg^{\top},s)$ satisfies the functional equation
\[
\pi^{-s}\Gamma(s)Z_r(gg^{\top},s)=\pi^{-\left(\frac{r}{2}-s\right)}\Gamma\left(\frac{r}{2}-s\right){Z_r\left((gg^\top)^{-1},\frac{r}{2}-s\right).}
\]
\end{corollary}

\subsection{Global automorphic Sobolev spaces}
\label{Sobolev}
Global automorphic Sobolev spaces provide a reasonable context to discuss solutions to partial differential equations in automorphic forms. We briefly recall results here and refer the interested reader to \cite{DeC21} and \cite{Gar18}*{Chapter 12}.

Let $\Delta$ be the invariant Laplacian on $\Gamma\backslash G/K$ descended from Casimir {(see Section \ref{Laplacian})}. For $\Xi$ a locally compact, Hausdorff, $\sigma$-compact topological space parametrizing cuspforms, Eisenstein series, and their residues, and $d\xi$ a positive regular Borel measure on $\Xi$, we have $L^2$-expansions
\[
f=\int_{\Xi}\langle f, \Phi_\xi\rangle\cdot \Phi_\xi\hspace{1mm} d\xi
\]
and Plancherel
\[
\left| f \right|^2_{L^2}=\int_{\Xi}\left|\langle f,\Phi_\xi\rangle\right|^2d\xi,
\]
at first for test functions $C_c^\infty(\Gamma\backslash G/K)$, then extended to $L^2(\Gamma\backslash G/K)$ by continuity.

Many of the $\Delta$-eigenfunctions $\Phi_\xi$ appearing in the $L^2$-decomposition are not in $L^2(\Gamma\backslash G/K)$, but they are all smooth, and we can arrange that $\xi\to \Phi_\xi$ is a continuous $C ^\infty(\Gamma\backslash G/K)$-valued function on $\Xi$. Since integration of elements of $C^\infty(\Gamma\backslash G/K)$ against a fixed test function $f\in C_c^\infty(\Gamma\backslash G/K)$ is a continuous linear functional on $C^\infty(\Gamma\backslash G/K)$,
\[
\xi\mapsto \langle f, \Phi_\xi\rangle\hspace{5mm}(\text{for fixed } f\in C_c^\infty(\Gamma\backslash G/K))
\]
is a continuous $\mathbb{C}$-valued function on $\Xi$ and thus has pointwise values. The implied integrals in the spectral expansion do not converge for all $f$ in $L^2$ (and $L^2$-expansions do not reliably converge pointwise), but Plancherel asserts that the literal integrals $f\to (\xi\to \langle f, \Phi_\xi\rangle)$ on test functions extend to an isometry $\mathcal{F}: L^2(\Gamma\backslash G/K)\to L^2(\Xi)$ and that the spectral coefficients $\langle f,\Phi_\xi\rangle$ for $f$ in $L^2$ are extensions by continuity of the literal integrals for test functions.

On the other hand, Plancherel does not assert anything about pointwise values of cuspforms or Eisenstein series, or about residues of Eisenstein series. Even in the simplest cases, as was known to Hecke and Maass in the {early 20th century}, {certain linear combinations of} pointwise values of $E_s$ are essentially $\zeta_k(s)/\zeta_k(2s)$ for complex quadratic extensions $k$ of $\mathbb{Q}$, and sharp pointwise bounds on the critical line would imply Lindel\"of \cite{Gar09}.

In a suitable global automorphic Sobolev space and as described precisely in what follows, $\Delta$ differentiates spectral expansions of test functions term-wise; then, integration by parts gives 
\[
\Delta f=\int_\Xi\langle\Delta f, \Phi_\xi\rangle\Phi_\xi\hspace{1mm} d\xi=\int_{\Xi}\lambda_\xi\cdot\langle f, \Phi_\xi\rangle\Phi_\xi\hspace{1mm} d\xi\hspace{10mm}(\text{for }f\in C_c^\infty(\Gamma\backslash G/K)).
\]
For each non-negative integer $k$, we define an inner product $\langle f,g\rangle_{\B^k}$ on $C_c^\infty(\Gamma\backslash G/K)$ by
\[
\langle f,g\rangle_{\B^k}=\langle (1-\Delta)^k f,g\rangle=\int_{\Gamma\backslash G/K}(1-\Delta)^k f\cdot \bar g,
\]
and the $k$-th global automorphic Sobolev space $\B^k$ is the completion of $C_c^\infty(\Gamma\backslash G/K)$ with respect to the norm $|f|_{\B^k}=\langle f,f\rangle^{1/2}_{\B^k}$.

Since $P(-\Delta)$ is non-negative on test functions for any polynomial $P$ with non-negative real coefficients, $|f|_{\B^{k+1}}\geq |f|_{\B^k}$ for test functions, and there is a canonical continuous injection $\B^{k+1}\to \B^k$ with dense image {\cite{Gar18}*{Claim 12.3.9}}.  {By design, $\Delta: C_c^\infty(\Gamma\backslash G/K)\to C_c^\infty(\Gamma\backslash G/K)$ is continuous when the source is given the $\B^{k+2}$ topology and the target is given the $\B^k$ topology}, so $\Delta$ extends by continuity to a continuous linear map
\[ 
\Delta: \B^{k+2}(\Gamma\backslash G/K)\to \B^k(\Gamma\backslash G/K) \hspace{5mm}(\text{for }0\leq k\in\mathbb{Z}).
\] 

We characterize negative-index Sobolev spaces as the Hilbert space duals $\B^{-k}=(\B^k)^*$ and identify $\B^0=L^2(\Gamma\backslash G/K)$ with its dual via the {Riesz-Fr\'echet map} $\Lambda: f\to \langle -,f\rangle$ composed with pointwise conjugation $c: f\to \overline{f}$. Since the {continuous inclusions $i:\B^{k+1}\hookrightarrow \B^k$ have dense image}, their adjoints $i^*:\B^{-k}\hookrightarrow \B^{-k-1}$ are also continuous injections with dense image. Thus, we have a chain of continuous linear maps between Hilbert spaces:
\[
\begin{tikzcd}
... \arrow[r,"i"] & \B^2 \arrow[r,"i"] & \B^1 \arrow[r, "i"] & \B^0 \arrow{r}{\Lambda\circ c}[swap]{\approx} & (\B^0)^* \arrow[r,"i^*"] & \B^{-1}  \arrow[r,"i^*"] & \B^{-2} \arrow[r,"i^*"] & ...
\end{tikzcd}
\]
Continuous $L^2$-differentiation for positive-index Sobolev spaces gives an extension by continuity from $\Delta: C_c^\infty(\Gamma\backslash G/K)\to C_c^\infty(\Gamma\backslash G/K)$, equipped with the $\B^{-k}$ and $\B^{-k-2}$ topologies, to a continuous map $\Delta: \B^{-k}\to\B^{-k-2}$ on negative-index Sobolev spaces, with the extension to $\Delta: \B^1\to \B^{-1}$ characterized by
\[
((1-\Delta)f)(F) = \langle f,\bar F\rangle_{\B^1},\hspace{5mm}\text{for }f,F\in \B^1.
\]

To bookend the Sobolev tower, consider
\[
\B^\infty=\lim_{0\leq k\in\mathbb{Z}}\B^{2k},\hspace{5mm}\B^{-\infty}=\text{colim}_{\hspace{1mm}0\leq k\in\mathbb{Z}}\B^{-2k}.
\]
{Since every continuous linear functional on a limit of Banach spaces factors through some limitand when the image of the limit is dense in the limitands} {(see \cite{Gar17a})}, $\B^{-\infty}$ is the dual of $\B^\infty$.  By characterization of projective limits, there is a unique, continuous map $\Delta: \B^{\infty}\to \B^{\infty}$ induced from the family of compatible maps $\B^\infty\hookrightarrow \B^{k+2}\to \B^{k}$. Similarly, there is a unique, continuous map $\Delta: \B^{-\infty}\to \B^{-\infty}$ which is compatible with the family of inclusions $\B^{-k}\to\B^{-k-2}\hookrightarrow \B^{-\infty}$, by characterization of the colimit.

For $0\leq k\in \Z$ and $\mathcal{F}: L^2(\Gamma\backslash G/K)\to L^2(\Xi)$ the spectral transform, certainly $\mathcal{F}\B^k$ is contained in  $V^k$, where 
\[
{V^s= \{\text{measurable }v\text{ on }\Xi: (1-\lambda_{\xi})^{s/2}\cdot v\in L^2(\Xi)\}},\hspace{5mm}\text{for }s\in \R.
\]
For any $s\in \R$, we can endow $V^s$ with a Hilbert-space structure from the norm
\[
|v|^2_{V^s}=\int_{\Xi}(1-\lambda_{\xi})^{s}|v(\xi)|^2\hspace{1mm}d\xi.
\]
As one might expect, there is a continuous inclusion $V^s\to V^t$ with dense image for $s>t$, and the space $V^{-s}$ is naturally the Hilbert space dual $(V^s)^*$ of $V^s$, with Hermitian pairing given by the asymmetrical extension of the Hermitian pairing $V^0\times V^0\to \C$ by $v\times w\to \langle v,w\rangle_{V^0}$:
\[
\langle v,w\rangle_{V^s\times V^{-s}} = \int_{\Xi}v(\xi)\hspace{1mm}\overline{w(\xi)}\hspace{1mm}d\xi.
\]

The spectral transform $\mathcal{F}: C_c^\infty(\Gamma\backslash G/K)\to L^2(\Xi)$ induces Hilbert space isomorphisms $\mathcal{F}: \B^{2k}\to V^{2k}$ for all non-negative integers $k$. The case $k=0$ is the Plancherel theorem. Let $M_\xi(v)=(1-\lambda_\xi)\cdot v$ be the multiplication operator. We observe that $\mathcal{F}$ intertwines the operators $1-\Delta: \B^{2k}\to \B^{2k-2}$ and $M_\xi: V^{2k}\to V^{2k-2}$ for Sobolev spaces of positive index; furthermore, dualization gives the same result on negative-index spaces. Thus, the following diagram commutes for $1\leq k\in \Z$:
\[
\begin{tikzcd}[column sep=large]
\B^{2k}\arrow{r}{1-\Delta}[swap]{\approx} \arrow{d}{\approx}[swap]{\mathcal{F}} & \B^{2k-2} \arrow{d}{\approx}[swap]{\mathcal{F}}\\
V^{2k}\arrow{r}{\approx}[swap]{M_\xi} & V^{2k-2}.
\end{tikzcd}
\]
This allows us to define an isomorphism $\mathcal{F}: \B^{-2k}\to V^{-2k}$ as the adjoint to the isomorphism $\mathcal{F}^{-1}: V^{2k}\to \B^{2k}.$

The odd-index case is slightly complicated since $\B^{2k+1}$ is not naturally isomorphic to $\B^0$, but there is no issue since $\mathcal{F}: \B^1\to V^1$ is an isometry to its (dense) image and $\B^1$ is complete. The $\B^1\to \B^{-1}$ case follows since $\mathcal{F}^*\circ M_\xi\circ \mathcal{F}=1-\Delta$ on test functions, by the characterization of adjoints and de-symmetrized Plancherel.  As in the even-index case, by design, $(1-\Delta): \B^{2k+1}\to \B^{2k-1}$ and $M_{\xi}: V^{2k+1}\to V^{2k-1}$ are isomorphisms, intertwined by the Hilbert space isomorphism $\mathcal{F}: \B^{2k+1}\to V^{2k+1}$ induced from the spectral transform on test functions. Furthermore, we can define an isomorphism $\mathcal{F}: \B^{-1-2k}\to V^{-1-2k}$ as the adjoint to the isomorphism $\mathcal{F}^{-1}: V^{2k+1}\to \B^{2k+1}.$

Thus, for $k\in \Z$, the Hilbert space $V^k$ and $M_{\xi}$ are spectral-side mirrors of $\B^k$ and $1-\Delta$; furthermore, the spectral side allows us to define $\B^s$ for any $s\in\R$ as completion of test functions with respect to the norm $|f|_{\B^s} = |\mathcal{F}f|_{V^s}$. Sobolev imbedding makes explicit the classical regularity of functions in automorphic Sobolev spaces:

\begin{theorem}[Sobolev imbedding for automorphic functions]
For $r=\dim_\mathbb{R}\Gamma\backslash G/K$ and $s>\frac{r}{2}$, $\B^s\subset C^o(\Gamma\backslash G/K)$. Furthermore, for $f\in C_c^\infty(\Gamma\backslash G/K)$ and compact $C\subset \Gamma\backslash G/K$, we have $\sup_{z_o\in C}|f(z_o)|\ll_{C,s}|f|_{\B^s}$ and
\[
\lim_k\int_{\xi:|\lambda_\xi|\leq \ell}\langle f, \Phi_\xi \rangle\cdot \Phi_\xi \hspace{1mm}d\xi =f\hspace{5mm}(\text{in }C^o).
\]
\end{theorem}
\begin{proof}
The proofs in \cite{Gar18}*{Claim 12.3.19, Theorem 12.4.6} use a precursor to the trace formulas in \cite{Sel54}, \cite{Sel56}, \cite{Hej76}, and \cite{Iwa02}, namely the pre-trace formula
\[
\int_{\xi: |\lambda_\xi|\leq T^2} |\Phi_\xi(z_o)|^2\ll_C T^r\hspace{5mm}(\text{as }T\to \infty),
\]
to give a uniform bound on $\int_{\Xi}|\Phi_{\xi}(z_o)|^2\cdot (1-\lambda_\xi)^{-\frac{r}{2}-\varepsilon}$, for $z_o$ in compact $C\subset \Gamma\backslash G/K$:
\begin{align*}
&\int_{\Xi}|\Phi_{\xi}(z_o)|^2\cdot (1-\lambda_\xi)^{-\frac{r}{2}-\varepsilon}= \sum_{\ell = 1}^\infty \int_{\xi: \ell-1\leq |\lambda_\xi|< \ell} |\Phi_\xi(z_o)|^2 \cdot (1-\lambda_\xi)^{-\frac{r}{2}-\varepsilon}\\
&\ll_C \sum_{\ell = 1}^\infty \int_{\xi: |\lambda_\xi|< \ell} |\Phi_\xi(z_o)|^2 \cdot \left((1+\ell)^{-\frac{r}{2}-\varepsilon} - (1+(1+\ell))^{-\frac{r}{2}-\varepsilon}\right)\\
&\ll_C \sum_{\ell = 1}^\infty \int_{\xi: |\lambda_\xi|< \ell} |\Phi_\xi(z_o)|^2 \cdot (1+\ell)^{-\frac{r}{2}-\varepsilon-1} \ll_C \sum_{\ell = 1}^\infty \ell^{\frac{r}{2}} \cdot (1+\ell)^{-\frac{r}{2}-\varepsilon-1}<\infty.
\end{align*}
Since $\{\xi\in \Xi: |\lambda_{\xi}|\leq \ell\}$ is compact and $\xi\to \Phi_{\xi}$ is a continuous $C^\infty(\Gamma\backslash G/K)$-valued function on $\Xi$, the integrals
\[
f_\ell = \int_{\xi:|\lambda_\xi|\leq \ell}\langle f, \Phi_\xi \rangle\cdot \Phi_\xi(z) \hspace{1mm}d\xi ,\hspace{5mm}\text{for }f\in C_c^\infty(\Gamma\backslash G/K)
\]
exist as $C^\infty(\Gamma\backslash G/K)$-valued Gelfand-Pettis integrals, so are certainly continuous. By the spectral characterization of $\B^s$, $f_\ell \to f$ in the $\B^s$ topology.

For $s>\frac{r}{2}$, $f\in \B^s$, and $z_o\in C\subset \Gamma\backslash G/K$ compact, 
\begin{align*}
\left|\int_{\Xi} \mathcal{F}f(\xi)\cdot \Phi(z_o)\hspace{1mm}d\xi\right| &= \left|\int_{\Xi} \mathcal{F}f(\xi)(1-\lambda_\xi)^{s/2}\cdot (1-\lambda_\xi)^{-s/2}\Phi(z_o)\hspace{1mm}d\xi\right|\\
&\leq \left(\int_{\Xi} |\mathcal{F}f(\xi)|^2(1-\lambda_\xi)^{s}\hspace{1mm}d\xi\right)^{1/2}\cdot \left(\int_{\Xi} |\Phi(z_o)|^2\cdot (1-\lambda_\xi)^{-s}\hspace{1mm}d\xi\right)^{1/2} \\
&= |f|_{\B^s}\hspace{1mm}\cdot \left(\int_{\Xi} |\Phi(z_o)|^2\cdot (1-\lambda_\xi)^{-s}\hspace{1mm}d\xi\right)^{1/2}\ll_{C,s} |f|_{\B^s}
\end{align*}
by Cauchy-Schwarz-Bunyakowsky and the uniform bound above. That is, the $C^o$-norm (sup norm on compact) of the function $z_o\to \int_{\Xi} \mathcal{F}f(\xi)\cdot \Phi(z_o)\hspace{1mm}d\xi$ exists and is dominated by $|f|_{\B^s}$. Since $f_{\ell} \to f$ in $\B^s$ for a dense subset of $\B^s$ and $\B^s$ convergence implies $C^o$ convergence, $\B^s\subset C^o$ and, for $f\in \B^s$, $f^\ell\to f$ in $C^o$ (that is, pointwise and uniformly on compact).
\end{proof}

\subsection{Schwartz' Kernel Theorem, nuclearity}
\label{Schwartz}

% See http://www-users.math.umn.edu/~garrett/m/real/notes_2016-17/08c-nuclear_spaces_I.pdf

In Section \ref{Results}, and alternatively to \cite{HC66} and \cite{Cas89}, we construct a Schwartz space on $\Gamma\backslash G/K$ as a projective limit of Hilbert spaces with Hilbert-Schmidt transition maps. This limit is nuclear Fr\'echet, so a kernel theorem follows immediately for general categorical reasons.  Following \cite{Gar20}, we recall foundational results.

For $A,B,C$ all $k$-vector spaces \textit{without topologies} and $\otimes_k$ the usual tensor product of vector spaces, the Cartan-Eilenberg adjunction is 
\[
\text{Hom}_k(A,\text{Hom}_k(B,C))\approx \text{Hom}_k(A\otimes_k B,C)\hspace{5mm}(\text{by }\varphi(a\otimes b)\to \varphi(a)(b)),
\]
and the special case $C=k$ becomes
\[
\text{Hom}_k(A,B^*)\approx (A\otimes_k B)^*.
\]

For topological vector spaces $V$ and $W$, a genuine (categorical) tensor product is a topological vector space $X$ and a continuous linear map $j: V\times W\to X$ such that for every continuous bilinear map $K:V\times W\to Y$ to a topological vector space $Y$, there is a unique continuous linear map $k: X\to Y$ fitting into the commutative diagram
\[
\begin{tikzcd}
X \arrow[dashed,"\exists !\hspace{1mm}k"]{dr} \\
\arrow["j"]{u} V\times W \arrow[swap,"K"]{r} & Y.
\end{tikzcd}
\]
By an easy corollary of Baire category (see e.g. \cite{Gar20}*{Claim 15.1}), separately continuous maps on Hilbert spaces are also jointly continuous, so the lack of specification is justified.

When there is a genuine tensor product of topological vector spaces $V$ and $W$, Cartan-Eilenberg yields a Schwartz kernel theorem. However, there is no (categorical) tensor product in the category of Hilbert spaces and continuous linear maps \cite{Gar10}. While it is possible to put an inner product on the algebraic tensor product $V\otimes_{alg} W$, and the completion $V\otimes_{HS}W$ with respect to the associated norm is a Hilbert space (the space of Hilbert-Schmidt operators), this Hilbert space does not have the universal property in the categorical characterization of the tensor product, {since there are (jointly) continuous $\beta: V\times W\to X$ to Hilbert spaces $X$ which do not factor through any continuous linear map $B: V\otimes_{HS} W\to X$}. That is, not all continuous operators are Hilbert-Schmidt.

For example, let $V=\ell^2$ with orthonormal basis $\{v_i\}$, $W=V^*$ with the corresponding dual basis  $\lambda_i(v)=\langle v,v_i\rangle$, and $X=\mathbb{C}$. Since $\beta(v,\lambda)=\lambda(v)$ is (jointly) continuous, we might expect there is some continuous linear map $B: V\otimes_{HS} V^*\to \mathbb{C}$ such that for all coefficients $c_{ij}$
\[
B\left(\sum_{ij}c_{ij}v_i\otimes_{HS} \lambda_j\right)=\sum_{ij}c_{ij}\beta(v_i,\lambda_j)=\sum_ic_{ii}.
\]
However, $\sum_i\frac{1}{i}v_i\otimes_{HS}\lambda_i$ is in $V\otimes_{HS}V^*$, but the necessary value of $B$ is impossible since $\sum_i \frac{1}{i}$ diverges.

% nuclearity

Nuclear spaces have genuine tensor products \cite{Gar20}*{Theorem 7.5}, and countable projective (directed) limits of Hilbert spaces with Hilbert-Schmidt transition maps are the simplest types of nuclear spaces \cite{Gro55}: Let $V,W,V_1,W_1$ be Hilbert spaces with Hilbert-Schmidt maps $S: V_1\to V$ and $T:W_1\to W$. For any (jointly) continuous $\beta: V\times W\to X$ there is a unique continuous (and Hilbert-Schmidt) map $B:V_1\otimes_{HS}W_1\to X$ giving a commutative diagram
\[
\begin{tikzcd}
V_1\otimes_{HS}W_1\arrow{r} \arrow[bend left=60,"B", dashed]{drr} & V\otimes_{HS}W\\
V_1\times W_1 \arrow["S\times T"]{r} \arrow{u} & V\times W \arrow["\beta"]{r} \arrow{u} & X.
\end{tikzcd}
\]
Let $v_i$, $w_j$ be orthonormal bases for $V$ and $W$, respectively. By the continuity of $\beta$,
\[
\sum_{ij}\beta(Sv_i,Tw_j)\leq C\sum_{ij}|Sv_i|^2\cdot|Tw_j|^2=C\cdot|S|^2_{HS}\cdot |T|^2_{HS}<\infty.
\]
A continuous map $B: V_1\otimes_{HS}W_1\to X$ fitting into the diagram would have
\[
B\left(\sum_{ij}c_{ij}v_i\otimes w_j\right)=\sum_{ij}c_{ij}\beta(Sv_i,Tw_j),
\]
and since
\[
\sum_{ij}|c_{ij}\beta(Sv_i,Tw_j)|^2\leq \sum_{ij}|c_{ij}|^2\cdot \sum_{ij}|\beta(Sv_i,Tw_j)|^2\leq \sum_{ij}|c_{ij}|^2\cdot \left(C\cdot |S|^2_{HS}\cdot|T|^2_{HS}\right)
\]
by Cauchy-Schwarz-Bunyakowsky, $B$ is Hilbert-Schmidt (and exists).

The simplest examples of nuclear spaces are the projective limits $H^\infty(\mathbb{T}^m)$ of non-negative index Sobolev spaces on products of circles
\[
H^s(\T^m)=\left\{\sum_{\xi\in\Z^m} c_\xi e^{i\xi\cdot x}\in L^2(\T^m): |c_\xi|^2\dot(1 + |\xi|^2)^s<\infty\right\}\hspace{5mm}(\text{for }s\geq 0),
\]
originally defined as completions of smooth functions on $\T^m$. The Hilbert-Schmidt property is proven by a Rellich lemma, and 
the tensor product 
\[
H^\infty(\T^m)\otimes H^\infty(\T^n)\approx H^\infty(\T^{m+n})
\]
is constructed as
\[
\lim_{s} H^s(\T^m)\otimes \lim_{s} H^s(\T^n)=\lim_{s}\bigg(H^s(\T^m)\otimes_{HS}H^s(\T^n)\bigg).
\]

Then, Cartan Eilenberg yields the simplest example of Schwartz' Kernel Theorem: there is an isomorphism
\[
\text{Hom}^o(H^\infty(\T^m),H^{-\infty}(\T^n))\approx H^{-\infty}(\T^{m+n})
\]
where $\text{Hom}^o$ is the space of continuous linear maps and $H^{-\infty}$ is equipped with the weak-dual topology. The isomorphism is induced by
\[
(f\to (F\to \Phi(f\otimes F)))\leftarrow \Phi,
\]
and the distribution $\Phi\in H^{-\infty}(\T^{m+n})$ producing a given continuous map $H^\infty(\T^m)\to H^{-\infty}(\T^n)$ is the Schwartz kernel of the map.

\subsection{Friedrichs' extensions of restrictions of semi-bounded operators}
\label{Friedrichs}
Crucial to \cite{LP76}*{pp. 204-206}, \cite{CdV82}, \cite{BG20}, and Section \ref{Results} is Friedrichs' self-adjoint extension of semi-bounded operators on Hilbert spaces \cite{Fri34}.

For $T$ a symmetric, semi-bounded operator with domain $D$ dense in some Hilbert space $V$, Friedrichs' extension $\tilde T: \tilde D=\tilde T^{-1}V\to V$ is characterized by
\[
\langle (1+T)v,(1+\tilde T)^{-1}w\rangle=\langle v,w \rangle\hspace{5mm}\text{(for }v\in D, w\in V).
\]
Unlike the other self-adjoint extensions of symmetric operators classified by Stone \cite{Sto32} and vonNeumann \cite{vN30}, Friedrichs' extension has certain nice properties: for $V^1$ the completion of $D$ with respect to $\langle v,w\rangle_1=\langle(1+T)v,w\rangle$, the resolvent $(1+\tilde T)^{-1}: V\to V^1$ is continuous when $V^1$ has the finer topology and compact if the inclusion $V^1\to V$ is compact with the finer topology on $V^1$.

As in \cite{BG20} and \cite{Gar18}*{Section 11.2}, we can re-characterize Friedrichs' extension to facilitate finer analysis in terms of distributions and even meromorphically continue certain Eisenstein series beyond the critical line {(see Section \ref{Eisenstein})}:

Assume the Hilbert space $V$ has a complex conjugation map $c$ that commutes with $T$ such that the composition of the Riesz-Fr\'echet map is a complex-linear isomorphism between $V$ and its dual $V^*$, and let $j:V^1\to V$ be the continuous injection induced from the identity $D\to D$ with adjoint $j^*:V^*\to (V^{1})^*=V^{-1}$. The continuous linear map $T^\#=j^*\circ c\circ j:V^1\to V^{-1}$ given by
\[
T^\#(x)(y)=\langle x,\bar y\rangle_1\hspace{10mm}(\text{for }x,y\in V^1)
\]
is a topological isomorphism by Riesz-Fr\'echet, and we may use this map to more explicitly identify the domain $\tilde D=\tilde T^{-1}V$ of the Friedrichs extension $\tilde T$. Notice that $T^\#$ is an extension of $j^*\circ c\circ \tilde T$: for $x=\tilde T^{-1}x'\in \tilde D$ with $x'\in V$ and any $y\in V^1$,
\[
(T^\# x)(y)=\langle x,\bar y\rangle_1=((j^*\circ c)x')(y)=((j^*\circ c\circ \tilde T)x)(y),
\]
and for $x\in V^1$ such that $T^\#x=(j^*\circ c)y$ for some $y\in V$, and all $z\in V^1$,
\[
\langle z,\bar x\rangle_1=(T^\# x)(z)=((j^*\circ c)y)(z)=(cy)(jz)=\langle jz,\bar y\rangle=\langle z,\tilde T^{-1}\bar y\rangle_1.
\]
That is, the domain of $\tilde T$ is
\[
\tilde D:= \tilde T^{-1}V= \{x\in V^1: T^{\#}x\in (j^*\circ c)V\}.
\]

To legitimize the discussion in later sections, we discuss Friedrichs' extension of certain restrictions of $T$, following \cite{BG20}. Let $\Theta\subset D$ be a $T$-subspace which is stable under conjugation, $V_\Theta=\Theta^\perp$ the orthogonal complement to $\Theta$ in $V$, and $D_\Theta=D\cap V_\Theta$. Certainly $D_\Theta\subset V^1\cap V_\Theta$, so the $V^1$-closure of $D_\Theta$ is a subset of $V^1\cap V_\Theta$. However, $D_\Theta$ is not necessarily $V^1$-dense in $V^1\cap V_\Theta$. As in the cases of interest, we assume $D_\Theta$ is $V^1$-dense in $V^1\cap V_\Theta$. Under this assumption, the restriction $T_\Theta$ of $T$ to $D_\Theta$ is densely-defined and symmetric on $V_\Theta$ with Friedrichs extension $\tilde T_\Theta: \tilde D_\Theta\to V_\Theta$ and further extension 
\[
T_\Theta^\#: V^1\cap V_\Theta\to (V^1\cap V_\Theta)^*,
\]
characterized by 
\[
(T_\Theta^\# x)(y) = \langle x,\bar y\rangle_1\hspace{10mm}(\text{for }x,y\in V^1\cap V_\Theta).
\]
For $i_\Theta: V^1\cap V_\Theta\to V^1$ and $j_\Theta:V^1\cap V_\Theta\to V_\Theta$ the inclusions with adjoints $i^*_\Theta$ and $j_\Theta^*$, we have the following diagram:
\[
\begin{tikzcd}[column sep=large]
\tilde T\subset T^\#: &[-4em] V^1\arrow[r,"j"] & V\arrow[r,"j^*\circ\hspace{1mm}c"] &V^{-1}\arrow[d,"i^*_\Theta"]\\
\tilde T_\Theta\subset T^\#_\Theta: & V^1\cap V_\Theta\arrow[r,"j_\Theta"]\arrow[u,"i_\Theta"] & V_\Theta\arrow[r,"j_\Theta^*\hspace{1mm}\circ\hspace{1mm}c"]\arrow[u,"i"] & (V^1\cap V_\Theta)^*,
\end{tikzcd}
\]
which has a useful compatibility $T_\Theta^\#=i_\Theta^*\circ T^\#\circ i_\Theta$, since
\[
((i^*_\Theta\circ T^\#\circ i_\Theta)x)(y)=(T^\#\circ i_\Theta x)(i_\Theta y)=\langle i_\Theta x, i_\Theta \bar y\rangle_1=\langle x,\bar y\rangle_1=(T^\#_\Theta x)(y)
\]
for $x,y\in V^1\cap V^\Theta$. By injectivity of the inclusion $i:V_\Theta\to V$, 
\[
\tilde D_\Theta = \{x\in V^1\cap V_\Theta: T^\#_\Theta x \in (j_\Theta^*\circ c)V_\Theta \}= \{x\in V^1\cap V_\Theta: T^\#_\Theta x\in (i^*_\Theta\circ j^*\circ c)V\}.
\]
Furthermore, $T_\Theta^\# x= (i^*_\Theta\circ j^*\circ c)y$ for $y\in V$ if and only if $(T^\#\circ i_\Theta)x=(j^*\circ c)y+\theta$ for some $\theta\in \ker i^*_\Theta$, the $V^{-1}$ closure of $(j^*\circ c)\Theta$. 

\subsection{Perturbations of linear operators}
\label{Perturbations}

In his introduction to \cite{Kat66}, Kato credits Rayleigh \cite{Ray94} and Schr\"odinger \cite{Sch28} with the creation of perturbation theory of linear operators, the study of physical systems which deviate slightly from a simpler ideal system for which a (more) complete solution is known. Although such operators had been used by physicists to successfully predict physical phenomena (see e.g. \cite{Dir28}, \cite{Dir30}, \cite{Tho35}, \cite{BP35}), analytical considerations such as convergence and self-adjointness were not rigorously considered at least until a series of papers by Rellich, Friedrichs, Kato, and many others decades later. Self-adjointness of Hamiltonians $-\Delta + q$ was originally established by the Kato-Rellich theorem of 1939 and the Kato, Lions, Lax-Milgram, Nelson (KLMN) theorem. \cite{BF61}, \cite{Kat72}, \cite{Sim73}, \cite{Sim73b} and many others greatly broadened the class of potentials for which the corresponding Schr\"odinger operator was essentially self-adjoint.

 At the time, and currently, solvable models such as $(\Delta - \lambda)u=\delta(u)\cdot \delta$  were casually rewritten by physicists as $((-\Delta + \delta)-\lambda)u=0$, viewing $-\Delta +\delta$ as a perturbation of $-\Delta$ by a singular potential. Indeed,  as is now well understood, $(\Delta-\lambda)u=\delta(u)\cdot \delta$ can be rearranged to $(\Delta - \delta\otimes \delta-\lambda)u=0$.

Complementing the theory of singular potentials, \cite{Kat51} established a fairly complete theory for Hamiltonians $S=-\Delta + q$, with $q$ in one of four classes $L^2+L^{\infty}$, $L^2+(L^\infty)_{\varepsilon} = \{f: f=x + g_{\varepsilon}: x\in L^2, \norm{g_{\varepsilon}}_{\infty}<\varepsilon\}$, $L^2$, and $L^2\cap L^1$. \cite{Sim71} extended this theory beyond $L^2$ potentials to match physical phenomema by defining the sum of operators $-\Delta + q$ as a quadratic form, as suggested by Nelson and Faris, instead of as the sum of the Laplacian and a multiplication operator.  
\section{Applications of modern analysis of automorphic forms}
\label{Applications}
In Section \ref{DiffEQ}, we show how to use automorphic spectral expansions to solve differential equations on the Sobolev spaces defined in Section \ref{Sobolev}. Section \ref{GreenFn} demonstrates the utility of this viewpoint by computing the simplest automorphic Green's function. In Section \ref{ExoticEfns}, we recall the seemingly paradoxical Lax-Phillips discretization of the continuous spectrum of $\Delta$ and demonstrate how to identify the ``exotic" eigenfunctions essential to reprove meromorphic continuation of certain Eisenstein series in \cite{CdV82} and \cite{Gar18}, recounted in Section \ref{Eisenstein}. {Most provocatively, in Section \ref{Spacing}, we demonstrate how the natural self-adjoint operators of Colin de Verdi\`ere, Lax-Phillips, and Bombieri-Garrett have some implications for the spacing of zeros of $\zeta_k(s)$ on the critical line, for  $k=\Q(\sqrt d)$ quadratic fields with $d<0$.}

\subsection{Solving automorphic differential equations by division}
\label{DiffEQ}
%- PG Book Ch. 12.5, 12.6, PG vignette: self-adjoint operators on automorphic forms
We are now equipped to solve certain differential equations in automorphic forms.

Given $f\in \B^{-\infty}(\Gamma\backslash G/K)$ and $\lambda\in \C$, we want to solve 
\[
(\Delta-\lambda)u=f
\]
for $u\in \B^{-\infty}(\Gamma\backslash G/K)$, when possible. Applying the spectral transform $\mathcal{F}$ to both sides yields
\[
\mathcal{F}f = \mathcal{F}(\Delta-\lambda)u=(\lambda_\xi -\lambda)\mathcal{F}u,
\]
where both $\mathcal{F}f$ and $\mathcal{F}u$ are in weighted $L^2$ spaces on the spectral parameter space $\Xi$. {For complex $\lambda\notin(-\infty,0]$, the function $\lambda_\xi-\lambda$ is bounded away from 0 on $\Xi$}, and we can divide to obtain
\[
\mathcal{F}u = \frac{\mathcal{F}f}{\lambda_{\xi}-\lambda},
\]
such that
\[
u = \int_{\Xi}\frac{\mathcal{F}f(\xi)}{\lambda_\xi - \lambda}\cdot \Phi_\xi\hspace{1mm} d\xi\hspace{10mm}(\text{convergent in }\B^{-\infty})
\]
is a solution. In fact, this solution is unique in $\B^{-\infty}$, since any solution $v$ to the homogeneous equation would have spectral expansion $v=\int_{\Xi}\mathcal{F}v(\xi)\cdot\Phi_\xi \hspace{1mm}d\xi$ for $\mathcal F v=0$ almost everywhere.

For $\lambda\in (-\infty,0]$, solutions to $(\Delta-\lambda)u=f$ are expressible in terms of the eigenvalues $\lambda_\xi$ and eigenfunctions $\Phi_\xi$ {of $\Delta$}. Consider the simplest case $G=SL_2(\R)$, where $f\in \B^{-\infty}(\modcurve)$ has spectral expansion
\[
f = \sum_{\text{cfm }F}\langle f,F\rangle\cdot F + \frac{\langle f,1\rangle\cdot 1}{\langle1,1\rangle}+\frac{1}{4\pi i}\int_{(\frac{1}{2})}\langle f,E_s\rangle\cdot E_s\hspace{1mm} ds\hspace{10mm}(\text{convergent in }\B^{-\infty}),
\]
{where the indicated pairings and integrals are extensions by continuity of the literal pairings and integrals.} Denote $\lambda_s=s(s-1)$. For $\text{Re}(w)>\frac{1}{2}$, and $w\neq 1$, the equation $(\Delta-\lambda_w)u=f$ has solution
\[
u=\sum_{\text{cfm }F}\frac{\langle f,F\rangle \cdot F}{\lambda_{s_F}-\lambda_w}+\frac{\langle f,1\rangle\cdot 1}{(\lambda_1 - \lambda_w)\cdot \langle 1,1\rangle}+\frac{1}{4\pi i}\int_{(\frac{1}{2})}\frac{\langle f, E_s\rangle\cdot E_s}{\lambda_s-\lambda_w}ds,
\]
converging in at least $\B^{-\infty}$. For $f\in \B^r$, the spectral characterization shows that $u\in \B^{r+2}$ with the spectral expansion converging in the same Sobolev space.

To meromorphically continue the solution to $(\Delta - \lambda_w)u=f$ as a function of $w$, consider $u=u_w$ as a holomorphic or meromorphic {function-valued function} of $w$. By the spectral characterization of $\B^r$ and $\B^{r+2}$, the cuspidal component
\[
w\to u_w^{\text{cusp}}=\sum_{\text{cfm }F}\frac{\langle f,F\rangle\cdot F}{\lambda_{s_F}-\lambda_w}
\]
is visibly a meromorphic $\B^{r+2}$-valued function of $w\in  \C$, with poles at most at $w=s_F$, discrete points in $\C$ corresponding to the eigenvalues of cuspforms. The constant component is similar.

The continuous spectrum component
\[
w\to u_w^{\text{cts}}=\frac{1}{4\pi i}\int_{(\frac{1}{2})}\frac{\langle f, E_s\rangle\cdot E_s}{\lambda_s - \lambda_w} ds
\]
does not generally meromorphically continue as an $\B^{r+2}$-valued function: for $\lambda_w\leq -1/4$, if $(\Delta-\lambda_w)u=f$ has a solution $u\in \B^{-\infty}$, then $\langle f,E_w\rangle=0$ in the strong sense that $\langle f,E_s\rangle/(\lambda_s-\lambda_w)$ is locally integrable near $s=w$. However, we will show that for $X$ a quasi-complete, locally convex topological vector space containing $\B^{r+2}$ and Eisenstein series, the continuous spectrum component has a meromorphic continuation to $\C$ as an $X$-valued function of $w$, with functional equation
\[
u_w^{\text{cts}} = u_{1-w}^{\text{cts}}-\frac{\langle f,E_w\rangle\cdot E_w}{2w-1}.
\]

The following topological vector space contains both suitable global automorphic Sobolev spaces and Eisenstein series. Let $\mathcal{E}$ be the Fr\'echet space $C^\infty(\Gamma\backslash G/K)$ and $\mathcal{D}$ the space of test functions $C^\infty_c(\Gamma\backslash G/K)$. The quotient $X$ of $\B^r\oplus\hspace{1mm}\mathcal{E}$ by the closure of the anti-diagonal copy $\mathcal{D}^{-\Delta}=\{(\varphi,-\varphi): \varphi\in \mathcal{D}\}$ of $\mathcal{D}$ is a topological vector space fitting into the pushout diagram
\[
\begin{tikzcd}
\mathcal{D}\arrow[swap,d,"\text{inc}"] \arrow[r,"\text{inc}"] & \B^r \arrow[d,dashed]\\
\mathcal{E}\arrow[r,dashed] & X
\end{tikzcd}
\]
and is unique up to unique isomorphism. Furthermore, the seminorms 
\[
\nu_{\mu}(f)=\sup_{w\in K}\mu(f(w))\hspace{5mm}(\text{for }K\subset\Omega\text{ compact})
\]
give $X$ a quasi-complete, locally-convex topology.

We now prove the meromorphic continuation of the continuous component $u^{\text{cts}}_w$. At first for $\text{Re}(w)>\frac{1}{2}$,
\begin{align*}
u_w^{\text{cts}} &= \frac{1}{4\pi i}\int_{(\frac{1}{2})}\frac{\langle f, E_s\rangle\cdot E_s-\langle f,E_w\rangle\cdot E_w}{\lambda_s-\lambda_w}ds + \frac{1}{4\pi i}\int_{(\frac{1}{2})}\frac{\langle f,E_w\rangle\cdot E_w}{\lambda_s - \lambda_w}ds\\
&= \frac{1}{4\pi i}\int_{(\frac{1}{2})}\frac{\langle f,E_s\rangle\cdot E_s-\langle f,E_w\rangle\cdot E_w}{\lambda_s-\lambda_w}ds - \frac{\langle f, E_w\rangle\cdot E_w}{2(2w-1)}.
\end{align*}
The integral looks better behaved near $s=w$, but the appearance is misleading since it is not a literal integral. With $t=\text{Im}(s)$,
we can split the remaining integral at some height $T$:
\begin{align*}
\int_{(\frac{1}{2})}\frac{\langle f,E_s\rangle\cdot E_s-\langle f,E_w\rangle\cdot E_w}{\lambda_s-\lambda_w}\hspace{1mm}ds&= \int_{|t|>T}\frac{\langle f,E_s\rangle\cdot E_s-\langle f,E_w\rangle\cdot E_w}{\lambda_s-\lambda_w}\hspace{1mm}ds\\
&+ \int_{|t|\leq T}\frac{\langle f,E_s\rangle\cdot E_s-\langle f,E_w\rangle\cdot E_w}{\lambda_s-\lambda_w}\hspace{1mm}ds.
\end{align*}
The integral over $|t|>T$ is 
\[
\int_{|t|>T}\frac{\langle f,E_s\rangle\cdot E_s}{\lambda_s-\lambda_w}ds - \langle f,E_w\rangle E_w\cdot \int_{|t|>T}\frac{ds}{\lambda_s-\lambda_w}.
\]
While the first term is $\B^{r+2}$-valued, the second term takes values $X$, unless $\langle f, E_w\rangle$ vanishes. The integrand of
\[
\int_{|t|\leq T}\frac{\langle f,E_s\rangle\cdot E_s-\langle f,E_w\rangle\cdot E_w}{\lambda_s-\lambda_w}ds
\]
is a compactly supported, holomorphic $X$-valued function of the complex variables $s,w$ away from the diagonal $s=w$. By inspection of the vector-valued power-series expansion of the integral \cite{Gar18}*{Claim 15.8.1}, there is also cancellation on the diagonal, so the integrand is a holomorphic $X$-valued function of $s,w$. For any complex-analytic $X$-valued function $f(s,w)$ on some domain $\Omega_1\times \Omega_2\subset \C^2$, the function $s\to (w\to f(s,w))$ is a holomorphic $\text{Hol}(\Omega_2,X)$-valued function on $\Omega_1$. Thus for $\Omega$ an appropriate bounded open set containing the set where $|t|\leq T$, the integrand over $|t|\leq T$ is a compactly-supported, continuous $\text{Hol}(\Omega,X)$-valued function of $s$, and has a {Gelfand-Pettis integral} in $\text{Hol}(\Omega, X)$. That is, $w\to u_w^{\text{cts}}$ has a meromorphic continuation as an $X$-valued function of $w$.

Having proven its meromorphic continuation, we may use the integral expression of $u_w^{\text{cts}}$ to derive its functional equation. At first for $\text{Re}(w)<\frac{1}{2}$,
\begin{align*}
u_w^{\text{cts}}&= \frac{1}{4\pi i}\int_{(\frac{1}{2})}\frac{\langle f,E_s\rangle\cdot E_s-\langle f,E_w\rangle\cdot E_w}{\lambda_s-\lambda_w}ds - \frac{\langle f, E_w\rangle\cdot E_w}{2(2w-1)}\\
&= \frac{1}{4\pi i}\int_{(\frac{1}{2})}\frac{\langle f,E_s\rangle\cdot E_s}{\lambda_s-\lambda_w}ds - \frac{\langle f, E_w\rangle\cdot E_w}{(2w-1)} = u_{1-w}^{\text{cts}}- \frac{\langle f, E_w\rangle\cdot E_w}{(2w-1)},
\end{align*}
since $\text{Re}(1-w)>\frac{1}{2}$ when $\text{Re}(w)<\frac{1}{2}$. This identity extends to all $w\in \C$ away from poles by the identity principle.

\subsection{Computing the simplest automorphic Green's Function}
\label{GreenFn}

\noindent To illustrate the usefulness of the previous discussion, we compute the automorphic Green's function on the simplest quotient. \cite{Hub55} and \cite{Sel54} had independently considered such matters in the context of lattice-point problems in hyperbolic spaces. \cite{Neu73} considers the convergence and meromorphic continuation of an automorphic Green's function formed by winding up free-space Green's functions. Complications or failure of spectral expressions to converge pointwise can be avoided by considering convergence in suitable global automorphic Sobolev spaces, as in \cite{Gar18}*{Example 12.6}, \cite{DeC12}, \cite{DeC16}, \cite{DeC21} and Section \ref{Sobolev}.

Since $\B^s\subset C^o$ for any $s>\frac{r}{2}$, $f\to f(z)=\delta_{z}(f)$ is a continuous linear functional on $\B^s$. That is, $\delta_z\in \B^{-s}$ with spectral expansion
\[
\delta_{z}=\int_{\Xi} \overline{\Phi_\xi}(z)\cdot \Phi_\xi d\xi\hspace{5mm}(\text{convergent in }\B^{-s}).
\]
Thus, for $G=SL_2(\R)$, $\text{Re}(w)>\frac{1}{2}$ and $w\neq 1$, the equation $(\Delta - \lambda_w)u=\delta_z$ has a solution $u=u_w$ given by
\[
u=\sum_{\text{cfm }F}\frac{\overline{F(z)} \cdot F}{\lambda_{s_F}-\lambda_w}+\frac{1}{(\lambda_1 - \lambda_w)\cdot \langle 1,1\rangle}+\frac{1}{4\pi i}\int_{(\frac{1}{2})}\frac{E_{1-s}(z)\cdot E_s}{\lambda_s-\lambda_w}ds\hspace{5mm}(\text{convergent in }\B^{1-\varepsilon})
\]
for any $\varepsilon>0$. Furthermore, $w\to u_w$ has a meromorphic continuation in a topological vector space $X$, which includes $\B^{1-\varepsilon}$ and Eisenstein series.

We may also compute the constant term of this Green's function in a more direct style than \cite{Nie73} and \cite{Fay77}:

\begin{theorem} For $\text{Re}(w)>\frac{1}{2}$, $w\notin (\frac{1}{2},1]$, and $a\geq \text{Im}(z)$, the solution $u_w=u_{w,z}$ in $\B^{1-\varepsilon}$ of the equation $(\Delta - \lambda_w)u_w=\delta_z$ has constant term
\[
c_P u_w(ia) =\int_0^1 u_w(x+ia) dx = a^{1-w}\cdot \frac{E_w(z)}{1-2w},
\]
where $z=x+iy$ on $G/K\approx\mathfrak{H}$.
\end{theorem}
\begin{proof}
We recall the proof of \cite{Gar18}*{Theorem 12.6.1}. The orbits of $(N\cap \Gamma)\backslash N$ are compact and codimension 1, so the distribution 
\[
\eta_a f = c_P f(ia) = \int_0^1 f(x+ia)dx
\]
which evaluates the constant term at height $a$ is in $\B^{-\frac{1}{2}-\varepsilon}$ for every $\varepsilon > 0$. Since $u_w\in \B^{1-\varepsilon}$ for every $\varepsilon>0$, $\eta_a$ gives a continuous linear functional on a Sobolev space containing $u_w$, and by the extended asymmetrical form of Plancherel,
\begin{align*}
\eta_a(u_w) &= \int_{\Xi}\mathcal{F}\eta_a\cdot \overline{\mathcal{F}u_w}= \frac{\eta_a(1)\cdot \delta_z(1)}{(\lambda_1 - \lambda_w)\langle 1,1\rangle} + \frac{1}{4\pi i}\int_{(\frac{1}{2})} \frac{\eta_a E_{1-s}\cdot \delta_zE_s}{\lambda_s-\lambda_w} ds\\
&= \frac{1}{(\lambda_1 - \lambda_w)\langle 1,1\rangle} + \frac{1}{4\pi i}\int_{(\frac{1}{2})} \frac{(a^{1-s}+c_{1-s}a^s)\cdot E_s(z)}{\lambda_s-\lambda_w} ds.
\end{align*}
By the functional equation for the Eisenstein series and the change of variables $s$ to $1-s$, the integral of $c_{1-s}a^s\cdot E_s(z)/(\lambda_s-\lambda_w)$ produces a copy of the integral of $a^{1-s}\cdot E_s(z)/(\lambda_s-\lambda_w)$ such that
\[
\eta_a(u_w) = \frac{1}{(\lambda_1 - \lambda_w)\langle 1,1\rangle} + \frac{1}{2\pi i}\int_{(\frac{1}{2})} \frac{a^{1-s}\cdot E_s(z)}{\lambda_s-\lambda_w} ds. 
\]
From the theory of the constant term \cite{Gar18}*{Theorem 8.1.1}, the Eisenstein series $E_s(z)$ is asymptotically dominated by its constant term $y^s+c_sy^{1-s}$. For $a\geq y$ and by elementary estimates, the contour $\text{Re}(s)=\frac{1}{2}$ can be pushed indefinitely to the right, picking up residues at $s=1$ (due to the pole of $E_s$) and at $s=w$ (due to the denominator). The constant-function term exactly cancels the residue at $s=1$ and since $\lambda_s-\lambda_w=(s-w)(s+w-1)$,
\[
\eta_a(u_w)= -\text{Res}_{s=w}a^{1-s}\cdot \frac{E_s(z)}{\lambda_s-\lambda_w}=-a^{1-w}\cdot \frac{E_w(z)}{w-1+w}.
\]
This expression gives the meromorphic continuation of $w\to \eta_a(u_w)$, which is the constant term of the meromorphic continuation of $u_w$ to a larger topological vector space $X$, assuming the topology on $X$ is fine enough such that $\eta_a$ is a continuous linear functional on it.
\end{proof}

\subsection{Partial discretization of the continuous spectrum}
\label{ExoticEfns}
Friedrichs' extension of $\Delta$ restricted to test functions in the Lax-Phillips space of $L^2$ pseudo-cuspforms has purely discrete spectrum, consisting of genuine cuspforms and infinitely-many ``exotic" eigenfunctions which are not eigenfunctions of $\Delta$ \cite{LP76}*{pp. 204-206}. The seemingly paradoxical discretization of the continuous spectrum of $\Delta$ is essential to the proof of meromorphic continuation of the simplest Eisenstein series in \cite{CdV82}, and \cite{Gar18}*{Section 11.11} {generalizes the method to prove meromorphic continuation of Eisenstein series for maximal proper parabolic subgroups $P^{r_1,r_2}\subset GL_{r_1+r_2}(\R)$, as an alternate proof to \cite{MW89}*{Appendix}} and Theorem \ref{EpsteinMero} in Section \ref{Degenerate}.

The notion of pseudo-cuspforms is simplest for unicuspidal quotients $\Gamma\backslash G/K$, namely
\[
L^2_a(\Gamma\backslash G/K) = \{f\in L^2(\Gamma\backslash G/K): c_P f(g)=0 \text{ for } \eta(g)\geq a\},
\]
{for $P$ any standard parabolic and $\eta(n_xm_yk)$ a suitable notion of height. Vanishing of the constant term $c_Pf$ above height $a$ is equivalent to orthogonality $ \langle f, \Psi_\varphi\rangle_{L^2}=0$ to pseudo-Eisenstein series $\Psi_\varphi$ with test-function data $\varphi\in C_c^\infty(0,\infty)$ supported on $[a,\infty)$.}

While density of $D_a=C_c^\infty(\Gamma\backslash G/K)\cap L^2_a(\Gamma\backslash G/K)$ in $L^2_a$ is not obvious, $D_a$ is provably dense in $L^2_a(\Gamma\backslash G/K)$ for $a\gg 1$ \cite{Gar18}*{Lemma 10.3.1}. The proof approximates $f\in L^2_a(\Gamma\backslash G/K)$ by test functions, then uses the $a\gg1$ condition to consider well-behaved smooth cut-offs of the constant term near height $a$, with the width of the cut-off region shrinking to 0. The averaged action images $\psi_n\cdot f$ are smooth, but the averaging smears the support of the constant term, depending on the support of $\psi_n$. The condition $a\gg 1$ ensures that the standard Siegel set $\mathfrak{S}_a$ has the property that $\mathfrak{S}_a\cap\gamma\mathfrak{S}_a\neq 0$ implies $\gamma\in \Gamma\cap P$ such that the natural smooth cutting-off of the constant term near the given height interacts with the constant-term vanishing in a controlled manner. Such $a\gg 1$ exists by reduction theory \cite{Gar18}*{Section 1.5}.
While ensuring the required density, the $a\gg 1$ condition means that the proof mechanism does not immediately yield discreteness of genuine cuspforms $L^2_o$. However, it is possible to salvage with a little work \cite{Gar18}*{Theorem 7.1.1}.

The restriction $\Delta_a$ of $\Delta$ to $D_a$ is symmetric, semi-bounded, and densely defined, so it has a self-adjoint Friedrichs extension $\tilde\Delta_a$. Let $\B_a^1$ be the completion of $D_a$ with respect to the Sobolev-like norm
\[
|f|^2_{\B_a^1}=\langle(1-\Delta)f,f\rangle_{L^2}\hspace{5mm}(\text{for }f\in D_a).
\]
Since the inclusion $L^2_a\hookrightarrow \B_a^1$ is compact when $\B_a^1$ has the finer topology \cite{LP76}*{pp. 204-206}, \cite{Gar18}*{Sections 10.7-10.8}, $\tilde \Delta_a$ has compact resolvent on $L^2_a(\Gamma\backslash G/K)$ as the composition of a continuous map and compact inclusion. {Proof of compactness of the inclusion relies on tail estimates \cite{LP76}*{pp. 204-206}, \cite{Gar18}*{Sections 10.4-10.6}}. In the simplest case, the crucial estimate is as follows. Given $\varepsilon>0$, a cut-off $c\geq a$ can be made sufficiently large so that the image of the $\B_a^1$-unit ball $B$ lies in a single $\varepsilon$-ball in $L^2(\modcurve)$. That is, for $f\in \B_a^1$,
\[
\lim_{c\to \infty}\int_{y>c}|f(z)|^2\frac{dxdy}{y^2}\to 0\hspace{5mm}(\text{uniformly for }|f|_{\B^1}\leq 1).
\]

The seeming paradox is that $L^2_a(\Gamma\backslash G/K)$ contains the space of genuine $L^2$-cuspforms and an infinite-dimensional space of pseudo-Eisenstein series. {For example, for $a'<a$ with $a'$ still large enough such that $\gamma\mathfrak{S}_{a'}\cap\mathfrak{S}_{a'}\neq \phi$ implies $\gamma\in \Gamma\cap P$ and for $\varphi\in C_c^\infty(0,\infty)$ supported on $[a',a]$, the pseudo-Eisenstein series $\Psi_\varphi(g)=\sum_{\gamma\in \Gamma}\varphi(\gamma g)$ is identically 0 in the region $\eta(g)>a$}. By the standard spectral decomposition, these pseudo-Eisenstein series are integrals of Eisenstein series, so some part of the continuous spectrum of $\Delta$ becomes discrete for $\tilde\Delta_a$.

Identification of exotic eigenfunctions is also non-trivial. In the simplest case $\modcurve$, the truncated Eisenstein series $\wedge^a E_s$ is an eigenfunction for $\tilde\Delta_a$ if and only if $c_P E_w(x+ia)=0$, and the spectral characterization of the global automorphic Sobolev spaces $\B^s$ shows that all eigenvalues $\lambda_w<-\frac{1}{4}$ arise this way. That is, all eigenfunctions for $\lambda_w<-\frac{1}{4}$ are truncated Eisenstein series \cite{LP76}*{pp. 204-206}, \cite{BG20}. These truncated Eisenstein series are in $L^2$ by the theory of the constant term \cite{Gar18}*{Claim 1.11.3}, but are not smooth, and therefore not eigenfunctions of the elliptic differential operator $\Delta$.

We illustrate an identification method in the simplest case, following \cite{CdV82} with detailed explanation in \cite{Gar18}. Take $a\gg 1$ and use the coordinates $z=x+iy\in\mathfrak{H}$. First we show $(\tilde \Delta_a -\lambda_w)\wedge^a E_w=0$. That is, we will show $(\Delta-\lambda_w)\wedge^a E_w = c\cdot \eta_a$ for some constant $c=c(w,a)$.  Away from $y=a$, $(\Delta -\lambda_w)\wedge^a E_w(z)=0$ locally. In $y\gg 1$, $\Delta$ annihilates all Fourier components of $E_w$ but the constant term, and in the lower part of a sufficiently large Siegel set, the operator also annihilates the constant term. To compute near $y=a$, let $H$ be the Heaviside function $H(y)=0$ for $y<0$ and $H(y)=1$ for $y>0$. Near $y=a$ (and with all equalities and derivatives in $\B^{-1}$ sense), 
\begin{align*}
(\Delta - \lambda_w)\wedge^a E_w(z) &= (\Delta-\lambda_w)\left(H(a-y)\cdot (y^w + c_w y^{1-w})\right)\\
&= (y^2 \frac{\partial^2}{\partial y^2} - w(w-1))\left(H(a-y)\cdot (y^w + c_w y^{1-w})\right)\\
&=y^2\left(\delta'_a\cdot(y^w+c_wy^{1-w})-2\delta_a\cdot(wy^{w-1}+(1-w)c_wy^{-w})\right).
\end{align*}
For $y^w+c_wy^{1-w}=0$, the term with $\delta'_a$ vanishes, and for $y=a$, the rest simplifies to
\[
-2a\delta_a\cdot(wa^{w}+(1-w)c_wa^{1-w})=-2\delta_a\cdot (2w-1)a^{w+1}.
\]
On functions of $y$ independent of $x$, this is $2(2w-1)a^{w+1}\cdot\eta_a\in\B^{-1}$. If $a^w+c_wa^{1-w}\neq 0$, the term with $\delta_a'$ remains, and the expression is not inside $\B^{-1}$ nor in the domain of $\tilde \Delta_a$. That is, $\wedge^a E_w$ is an eigenfunction only if $a^w + c_wa^{1-w}=0$.

For $G=SL_r(\R)$, the notion of pseudo-cuspform is more complicated. Let $A$ be the standard maximal torus of diagonal real matrices
\[
A=\left\{\begin{pmatrix}
a_1 & \\
& \ddots\\
& & a_r
\end{pmatrix}: a_i\in \R
\right\},
\]
$A^+$ the subgroup of positive real diagonal matrices, and $\Phi=\{\alpha_i(a)=\frac{a_i}{a_{i+1}}: i=1,...,r-1\}$ a choice of positive simple roots. By reduction theory \cite{Gar18}*{Section 3.3}, there is a sufficently large Siegel set $\mathfrak{S}=\mathfrak{S}_{t_o}=\{nmk: n\in N^{\text{min}}, m\in A^+, k\in K, \alpha(m)\geq t_o\text{ for all }\alpha\in\Phi\}$ associated to the standard minimal parabolic $P^{\text{min}}$ such that $\Gamma\cdot \mathfrak{S}=G$. For real $a\gg1$, and $Y_a\subset \mathfrak{S}$ given by
\[
Y_a=\{nmk\in \mathfrak{S}: \alpha(m)\geq a\text{ for some }\alpha\in \Phi\},
\]
the relevant space of pseudo-cuspforms is $V_a$, the $L^2$-closure of
\[
D_a=\{f\in C_c^\infty(\Gamma\backslash G/K): \text{ for }g\in Y_a, c_Pf(g)=0,\text{ for all standard parabolics } P\}.
\]
This space is certainly a subset of
\[
L^2_a=\{f\in L^2(\Gamma\backslash G/K):\text{ for }g\in Y_a,c_Pf(g)=0,\text{ for all standard parabolics } P\},
\]
but density of $D_a$ in the bigger space $L^2_a$ is not necessary for proving meromorphic continuation of cuspidal-data Eisenstein series.

As in the simplest case, we can prove that for sufficiently high cut-off heights $\eta$, there must be infinitely-many eigenfunctions for $\tilde\Delta_a$ that were not eigenfunctions for $\Delta$. {For example, let $P=P^{r,r}\subset SL_{2r}$ be a self-associate maximal proper parabolic, $f$ a cusp form on the Levi component $M=M^P$, and $E_{s,f}$ the corresponding cuspidal-data Eisenstein series with constant term $c_P E_{s,f}$. For $A^P$ the center of $M$, $M^1$ the subgroup of $M$ consisting of matrices in $r$-by-$r$ blocks $\begin{pmatrix} a & 0 \\ 0 & d\end{pmatrix}$ with $\det a = 1 = \det d$, and $s\in \C$ such that $c_P E_{s,f}(mm_1)=0$ for all $m_1\in M^1$ and $m\in A^P$ with $\alpha_r(m)=a$, the truncation $\wedge^a E_{s,f}$ is a $\tilde \Delta_a$ eigenfunction in $V_a$ \cite{Gar18}*{Claim 11.11.1}}.

\subsection{Meromorphic continuation of Eisenstein series}
\label{Eisenstein}

Here we recall an argument illustrating the utility of singular potentials.

Although Selberg had already proven meromorphic continuation of Eisenstein series on rank-one quotients in \cite{Sel56} (see also \cite{Roe56a},\cite{Roe56b}) and Langlands \cite{Lan76} had already treated the general case, Colin de Verdi\`ere \cite{CdV82} demonstrated how one can use extensions of restrictions of the Laplacian to give an alternate proof of meromorphic continuation of the simplest Eisenstein series (see also \cite{LP76}*{pp. 204-206}, \cite{Mu96}), with analytic continuation of $\zeta(s)$, as in \cite{Lan76}, \cite{Lan71}, and \cite{Sha78}, as a corollary. Furthermore, the method generalizes to prove meromorphic continuation of cuspidal-data Eisenstein series. {For a general result, see  \cite{MW95}*{Appendix}, and for maximal proper parabolics in $GL_r(\R)$, see \cite{Gar18}*{Sections 11.7-11.13}} or Theorem \ref{EpsteinMero} in Section \ref{Background}.

In short, compactness of the inclusion map $\B^1_a\hookrightarrow L^2_a$ for $a\gg 1$ yields compactness of the resolvent $(1-\tilde\Delta_a)^{-1}$ on $L^2_a$ and meromorphy of $a\to (1-\tilde\Delta_a)^{-1}$ as an operator-valued function taking values in the space of operators $T: L^2_a\to L^2_a$. Eisenstein series differ from truncated Eisenstein series by elementary functions, yielding meromorphic continuation.

First, we recall the argument for meromorphic continuation up to the critical line of the Eisenstein series on $\Gamma\backslash G/K=SL_2(\Z)\backslash \mathfrak{H}$, already known to Roelcke, et alia: For $\tilde \Delta$ the Friedrichs extension of the restriction of $\Delta$ to $D=C_c^\infty(\Gamma\backslash G/K)$ with domain $\tilde D$ and $\B^k$ the index $k$ global automorphic Sobolev space, we have the inclusions $\B^2\subset \tilde D\subset \B^1$. As usual, take $a\gg 1$. The quotient $\Gamma\backslash G/K$ is the union of a compact part $X_\text{cpt}$ and a geometrically simpler, non-compact part
\[
X_\infty =\Gamma_\infty\backslash\{x+iy\in \mathfrak{H}: y\geq a\}\approx (\Z\backslash \R) \times [a,\infty).
\]
Let $a<a'<a''$, and define a smooth cut-off function on $G/K\approx \mathfrak{H}$ given by
\[
\tau(g)=\tau(x+iy)=\begin{cases} 1 & y>a'\\
0 & y<a''.
\end{cases}
\]
We may define a pseudo-Eisenstein series $h_s$ by winding up the smoothly cut-off function $\tau(g)\cdot \text{Im}(g)^s$:
\[
h_s(g) = \sum_{\gamma\in \Gamma_\infty\backslash \Gamma}\tau(\gamma g)\cdot\text{Im}(\gamma g)^s. 
\]
Since $\tau$ is supported on $y\geq a''$, there is at most one non-vanishing summand in the expression for $h_s$ for any $g\in G/K$, and convergence is not an issue. Thus, the {function-valued function} $s\to h_s$ is entire.

Next, consider 
\[
\tilde E_s= h_s - (\tilde\Delta - \lambda_s)^{-1}(\Delta - \lambda_s) h_s.
 \]
 For all $s\in \C$, the smooth function $(\Delta-\lambda_s)h_s$ is supported on the image of the compact set $a\leq y\leq a'$ in $\Gamma\backslash G/K$ and is therefore in $C_c^\infty(\Gamma\backslash G/K)$. For $\lambda_s\notin (-\infty,0]$, that is for $s\in \C$ such that $\text{Re}(s)>\frac{1}{2}$ and $s\notin (\frac{1}{2},1]$, the resolvent $(\tilde \Delta-\lambda_s)^{-1}: L^2\to \B^1$ exists as an everywhere-defined, continuous operator.  That is, $s\to (\tilde\Delta - \lambda_s)^{-1}$ is a holomorphic operator-valued function on $S_\rho=\{s\in\C: \text{Re}(s)>\frac{1}{2}, s\notin \left(\frac{1}{2},1\right]\}$ such that $s\to\tilde E_s-h_s$ is a holomorphic $\B^1$-valued function on $S_\rho$. However, $s\to \tilde E_s-h_s$ is non-trivial. Since $h_s\notin L^2$ and $(\tilde\Delta - \lambda_s)^{-1}(\Delta-\lambda_s)h_s\in \B^1\subset L^2$ for $s\in S_\rho$, their difference $\tilde E_s$ cannot vanish on $S_\rho$.

 We claim $s\to \tilde E_s$ is a meromorphic $\B^1$-valued function on $S_\rho$ and gives a meromorphic continuation of $E_s$ to $\text{Re}(s)>\frac{1}{2}, s\notin (\frac{1}{2},1]$, and it suffices to prove that $u=\tilde E_s-h_s$ is the unique element of $\tilde D$ such that
 \[
 (\tilde \Delta - \lambda_s)u = -(\Delta - \lambda_s)h_s =: -H_s.
 \]
 Uniqueness follows from Friedrichs' construction, since $(\tilde \Delta-\lambda_s):\tilde D\to L^2(\Gamma\backslash G/K)$ is a bijection for $s\in S_\rho$. Furthermore, $\tilde E_s-h_s$ is in $C^\infty(\Gamma\backslash G/K) \cap L^2(\Gamma\backslash G/K)$ with
 \begin{align*}
 \Delta (\tilde E_s - h_s) &= (\Delta - \lambda_s)(\tilde E_s - h_s)+\lambda_s(\tilde E_s-h_s)\\
 &= (\Delta-\lambda_s)(\tilde\Delta-\lambda_s)^{-1}(\Delta-\lambda_s)h_s+\lambda_s(\tilde E_s - h_s)\\
 &=(\Delta-\lambda_s)h_s + \lambda_s(\tilde E_s-h_s)\in L^2,
 \end{align*}
 since $(\tilde\Delta-\lambda_s)^{-1}=(\Delta-\lambda_s)^{-1}$ on $S_\rho$ and $(\Delta-\lambda_s)h_s\in C_c^\infty(\Gamma\backslash G/K)$, the domain of $\Delta$. Thus, $\tilde E_s-h_s\in \B^2\subset \tilde D$, and the following computation is legitimate:
\begin{align*}
(\tilde\Delta - \lambda_s)(\tilde E_s-h_s) = (\tilde \Delta-\lambda_s)\left(h_s-(\tilde\Delta-\lambda_s)^{-1}H_s - h_s\right)=-H_s.
\end{align*}
We also have
\[
(\tilde \Delta - \lambda_s)(E_s-h_s) = -(\Delta-\lambda_s)h_s = -H_s,
\]
since $E_s$ is a $\Delta$-eigenfunction with eigenvalue $\lambda_s$. By uniqueness, $\tilde E_s - h_s$ = $E_s - h_s$ for $s\in S_\rho$, so $\tilde E_s$ gives an analytic continuation of $s\to E_s$ to $S_\rho$ as an $h_s + \B^1$-valued function.

Furthermore, Friedrichs' construction gives a bound for the $L^2$ norm of $E_s-h_s$. Since $(\Delta-\lambda_s)h_s$ has support in the image of $a\leq y \leq a'$,
\[
|(\Delta - \lambda_s)h_s|^2_2\leq \int_0^1 \int_a^{a'}(|\delta h_s|+|\lambda_sh_s|)^2\hspace{1mm}\frac{dxdy}{y^2}\ll_{a,a'}|\lambda_s|^2.
\]
Denote $\lambda_s=c+di$. For $s\in S_\rho$ and $v\in L^2$ arbitrary, since $\tilde\Delta$ is negative-definite and self-adjoint, 
\[
|(\tilde\Delta - \lambda_s)v|_2^2 = |(\tilde\Delta - c)v|_2^2 - i\cdot d\langle(\tilde\Delta -c)v,v\rangle_2 + i\cdot d\langle v, (\tilde\Delta -c)\rangle_2 + d^2|v|^2 {\geq} d^2|v|^2.
\]
Thus, for $s\in S_\rho$, the operator norm of the resolvent is estimated by 
\[
\norm{(\tilde \Delta - \lambda_s)^{-1}}^2\leq \frac{1}{d^2}=\frac{1}{2(\text{Re}(s)-\frac{1}{2})\cdot\text{Im}(s)}\hspace{5mm}(\text{for Re}(s)>\frac{1}{2}, \text{Im}(s)\neq 0),
\]
such that
\begin{align*}
|E_s-h_s|^2_2\leq \norm{(\tilde\Delta - \lambda_s)^{-1}}^2\cdot|(\Delta-\lambda_s)h_s|^2_2\ll_{a,a'}\frac{|\lambda_s|}{(\text{Re}(s)-\frac{1}{2})\cdot\text{Im}(s)}\hspace{1mm}.
\end{align*}

The simplest Eisenstein series has constant term $c_PE_s=y^s+c_sy^{1-s}$, for $\displaystyle c_s=\frac{\xi(2s-1)}{\xi(2s)}$ and $\xi(s)=\pi^{-s/2}\Gamma(s/2)\zeta(s)$ the completed zeta-function, so analytic continuation of $E_s$ to $\text{Re}(s)>\frac{1}{2}$ analytically continues $c_s$ to the same region, yielding the analytic continuation of $\zeta(s)$ to $\text{Re}(s)>0$ off the interval $[0,1]$.

The re-characterization of pseudo-Laplacians \cite{Gar18}*{Section 11.3} yields meromorphic continuation of Eisenstein series beyond the critical line and without the spectral decomposition of $L^2$ (which uses this meromorphic continuation). Referring to the notation of the last section, take $V=L^2(\Gamma\backslash G/K)$, $c$ pointwise conjugation, $D$ the space of automorphic test functions, $T=1-\Delta\big|_D$, $a\gg 1$ large enough such that $D_a=D\cap L^2_a$ is dense in the Lax-Phillips space of pseudo-cuspforms $L^2_a$, and $\Theta$ the space of pseudo-Eisenstein series $\Phi_\varphi$ with data $\varphi\in C_c^\infty(a,\infty)$. Notice $V_\Theta=\Theta^{\perp}=L^2_a$, $V^1$ is the global automorphic Sobolev space $\B^1$ with Hilbert space dual $\B^{-1}$, and the relevant diagram becomes
\[
\begin{tikzcd}[column sep=large]
\tilde \Delta\subset\Delta: &[-4em] \B^1\arrow[r,"j"] & L^2\arrow[r,"j^*\circ\hspace{1mm}c"] &\B^{-1}\arrow[d,"i^*_\Theta"]\\
\tilde\Delta_a\subset \Delta_a^\#: & \B^1\cap L^2_a\arrow[r,"j_\Theta"]\arrow[u,"i_\Theta"] & L^2_a\arrow[r,"j_\Theta^*\hspace{1mm}\circ\hspace{1mm}c"]\arrow[u,"i"] & (\B^1\cap L^2_a)^*.
\end{tikzcd}
\]

By spectral theory on multi-tori (and without the spectral decomposition of $\B^1$ or $\B^{-1}$), the evaluation of the constant term at height $a$ is in the $\B^{-1}$ closure of $\Theta$ for $a\gg 1$ \cite{Gar18}*{Section 11.3}. Furthermore, $D_a$ is $\B^1$-dense in $\B^1\cap L^2_a(\Gamma\backslash G/K)$ \cite{Gar18}*{Section 11.4}. Thus, $\tilde T_\Theta u = f $ for $f\in V_\Theta$ if and only if $u\in V^1\cap V_\Theta$ and $(T^\#\circ i_\Theta)u=(j^*\circ c\circ i)f+\theta$ for some $\theta\in \ker i^*_\Theta$. That is, $\tilde\Delta_a u=f$ for $f\in L^2_a(\Gamma\backslash G/K)$ if and only if $u\in \B^1\cap L^2_a(\Gamma\backslash G/K)$ and $\Delta u= f + c\cdot \eta_a$ for some constant $c$.

The proof of further meromorphic continuation starts as above, with $\tilde \Delta_a$ in place of $\tilde \Delta$. Let $h_s$ be the pseudo-Eisenstein series formed from the smooth cut-off $\tau\cdot y^s$ of $y^s$ and consider
\[
\tilde E_{a,s}=h_s - (\tilde \Delta_a - \lambda_s)^{-1}(\Delta-\lambda_s)h_s.
\]
Since $(\Delta-\lambda_s)h_s$ is compactly supported for all $s\in \C$, it is in $L^2_a$ for $a\gg1$ large enough. For $\lambda_s\notin (-\infty,0]$, $(\tilde\Delta_a-\lambda_s)^{-1}: L^2_a\to \tilde D_a$ is a bijection of $L^2_a$ to the domain of $\tilde\Delta_a$, so $u=\tilde E_{s,a}-h_s$ is the unique element of the domain of $\tilde\Delta_a$ satisfying
\[
(\tilde\Delta_a - \lambda_s)u = -(\Delta-\lambda_s)h_s.
\]
The map $s\to h_s$ is entire, so meromorphy of the resolvent $s\to (\tilde\Delta_a-\lambda_s)^{-1}$ for $s\in S_\rho$ yields the meromorphy of $s\to (\tilde E_{a,s}-h_s)$ as a $(\B^1\cap L^2_a)$-valued function in the same region.

Next we identify $c_P \tilde E_{a,s}$. Since $(\tilde \Delta_a-\lambda_s)^{-1}$ maps $(\Delta-\lambda_s)h_s$ to $L^2_a$, the constant term of $\tilde E_{a,s}$ is the constant term of $h_s$ above $y=a$, namely $y^s$. More generally,
\[
-(\Delta-\lambda_s)h_s=(\tilde\Delta_a - \lambda_s)(\tilde E_{s,a}-h_s)=(\Delta-\lambda_s)(\tilde E_{a,s}-h_s)+C_s\cdot\eta_a
\]
as distributions for some constant $C_s$. Rearranging yields a particularly useful identity:
\[
(\Delta - \lambda_s)\tilde E_{a,s} = -C_s\cdot \eta_a \hspace{5mm}(\text{as distributions}).
\]
Since $\Delta$ is $G$-invariant, it commutes with evaluation of the constant term, so the distribution $(\Delta-\lambda_s)c_P\tilde E_{a,s}$ vanishes away from $y=a$. On $0< y< a$, the distributional differential equation
\[
(y^2\frac{\partial^2}{\partial y^2}-\lambda_s)u = 0
\]
has solutions exactly of the form $A_s y^s + B^s y^{1-s}$ for constants $A_s, B_s$, so $c_P \tilde E_{a,s}$ must be of this form in $0<y<a$. Since $s\to \tilde E_{a,s}$ is meromorphic, so are $A_s,B_s$. In summary,
\[
c_P \tilde E_{a,s}=\begin{cases}
A_s y^s + B_s y^{1-s} & (0<y<a)\\
y^s & (y>a).
\end{cases}
\]
Furthermore, $c_P\tilde E_{a,s}$ continuous at $y=a$. By construction, $h_s$ is smooth, and $(\tilde\Delta_a-\lambda_s)^{-1}f\in \B^1$ for all $f\in L_a^2(\Gamma\backslash G/K)$. Thus $\tilde E_{a,s}$ is locally in $\B^1$ in the sense that $\psi\cdot \tilde E_{a,s}$ is in $\B^1$ for any smooth cut-off $\psi\in C_c^\infty(\Gamma\backslash G/K)$. In fact, we can choose a cutoff such that $c_P(\psi\cdot \tilde E_{a,s})\in \B^1\cap C^o(\Gamma\backslash G/K)$ by Sobolev imbedding on $\R$ \cite{Gar18}*{Section 11.5}.

Meromorphic continuation follows from using another pseudo-Eisenstein series. Let $\chi_{[a,\infty)}$ be the characteristic function of $[a,\infty)$ and
\[
\beta_{a,s}=\chi_{[a,\infty)}(y)\cdot (A_s y^s + B_s y^{1-s}-y^s).
\]
The support of $\beta_{a,s}$ is inside the set where $y\geq a$, so the corresponding pseudo-Eisenstein series $\Phi_{a,s}(g)=\sum_{\Gamma_\infty\backslash \Gamma}\beta_{a,s}(\gamma g)$ has as at most one non-zero summand for each $g\in G/K$ and therefore converges for all $s\in \C$. By \cite{Gar18}*{Section 11.5}, $A_s\cdot E_s = \tilde E_{a,s}+\Phi_{a,s}=B_s\cdot E_{1-s}$ for $A_s,B_s$ not identically 0. Thus, $s\to E_s$ meromorphically continues as a $(h_s+\B^1)$-valued function with functional equation $A_s\cdot E_s=B_s\cdot E_{1-s}$. Certainly, this functional equation reduces to the standard functional equation $E_s=c_sE_{1-s}$, and one can prove that $s\to E_s$ takes values in the space of smooth functions of moderate growth, by the theory of vector-valued integrals and Gelfand-Pettis corollaries \cite{Gar18}*{Section 11.5}. 

\subsection{Spacing of zeros of zeta}
\label{Spacing}
As appears in \cite{BG20}, exotic eigenfunction expansions also have implications for the spacing of zeros of $\zeta_k(s)$ on the critical line, for $k=\mathbb{Q}(\sqrt{d})$ quadratic fields with $d<0$.

For $a\gg_d 1$, let $\eta_a$ be evaluation of the constant term at height $a$ and $\theta$ (the restriction of) {the sum of Dirac deltas at the Heegner points attached to $k$}. For $u_w$ the meromorphically continued solution of $(\Delta-\lambda_w)u_w=\theta$ and $v_w$ that of $(\Delta-\lambda_w)v_w=\eta_a$, a linear combination $xu_w+yv_w$ is in $\ker\theta\cap\ker\eta_a$ if
 \[
 \begin{cases}
 0 = \theta(xu_w+yv_w) = \theta(u_w)x+\theta(v_w)y\\
 0 = \eta_a(xu_w+yv_w) = \eta_a(u_w)x+\eta_a(v_w)y
 \end{cases}
 \]
which has non-trivial solution if and only if
 \[
 \theta(u_w)\cdot\eta_a(v_w)-\eta_a(u_w)\cdot\theta(v_w)=0.
 \]
 From the spectral expansions and pairings, by residues,
 \[
 \eta_a(v_w)=\frac{a^w+c_wa^{1-w}}{1-2w}\cdot a^{1-w}.
 \]
Since $\theta E_w$ is essentially $\displaystyle\zeta_k(w)/\zeta(2w)$,
\[
\eta_a(u_w)=\theta(v_w)=\frac{\theta(E_w)a^{1-w}}{1-2w}=\frac{\zeta_k(w)a^{1-w}}{\zeta(2w)(1-2w)}.
\]
Any exotic eigenfunction of the Friedrichs extension $\tilde T_{a,\theta}$ of $\Delta$ restricted to $\ker\theta\cap \ker \eta_a$ is in $\ker\theta \cap \ker\eta_a$, so the identities above imply that $\lambda_w<-1/4$ is a $\tilde T_{a,\theta}$-eigenvalue if and only if
\[
\theta(u_w)\cdot\frac{a^w+c_wa^{1-w}}{1-2w}\cdot a^{1-w} - \left(\frac{\zeta_k(w)a^{1-w}}{\zeta(2w)(1-2w)}\right)^2=0.
\]

Next consider the Friedrichs extension $\tilde T_{\geq a}$ of $\Delta$ restricted to the space of pseudo-Eisenstein series with test function data and constant term vanishing at heights $y\geq a$. By a slight adaptation of the proof for $\tilde \Delta_a$, the operator $\tilde T_{\geq a}$ has purely discrete spectrum on the $L^2$-closure of $D_{\geq a}$, and for $\{f_n\}$ an orthogonal basis of $\tilde T_{\geq a}$-eigenfunctions, each $f_n\in L^2_a$ has eigenvalue $\lambda_{s_n}$ satisfying $a^{s_n}+c_{s_n}a^{1-s_n}=0$.

Let $j:\B^1\cap L^2_a\to \B^1$ be the inclusion map, $j^*: \B^{-1}\to (\B^1\cap L^2_a)^*$ its adjoint, and $\tilde T^{\#}_{\geq a}: \B^1\cap L^2_a\to (\B^1\cap L^2_a)^*$ the extension of $\tilde T_{\geq a}$ given by
 \[
 \tilde T^{\#}_{\geq a}(u)(v)=\langle u,\bar v\rangle_{\B^1}.
 \]
 Since $\theta$ is (the restriction of) a compactly supported real measure, $\theta \in \B^{-1}$. There is a spectral expansion convergent in $(\B^1\cap L^2_a)^*=j^*\B^{-1}$, a quotient of $\B^{-1}$:
 \[
 j^*\theta = \sum_n(j^*\theta)(f_n)\cdot f_n=\sum_n\theta(jf_n)\cdot f_n = \sum_n \theta(f_n)\cdot f_n\hspace{5mm}(\text{convergent in }j^*\B^{-1}).
 \]
 If $(\Delta -\lambda_w)u_w=\theta$ and $u_w\in \B^1\cap L^2_a$, then certainly $(\tilde T^{\#}_{\geq a}-\lambda_w)u_w = j^*\theta$. Noting that $j$ identifies $u_w$ with its image, this yields
 \[
 j^*\theta = (j^*\circ(\Delta-\lambda_w)\circ j)u_w=(\tilde T^{\#}_{\geq a}-\lambda_w)u_w.
 \]
 The latter equation can be solved by division, producing a spectral expansion
 \[
 u_w=\sum_{n}\frac{\theta(f_n)}{\lambda_{s_n}-\lambda_{w}}\cdot f_n\hspace{5mm}(\text{convergent in }\B^1\cap L^2_a\subset \B^1).
 \]
 
 By de-symmetrized Plancherel, the condition $\theta(u_w)=0$ is
 \[
 0=\theta(u_w)=\sum_{n}\frac{\theta(f_n)^2}{\lambda_{s_n}-\lambda_{w}}\hspace{1mm}.
 \]
By the intermediate value theorem, there is exactly one solution to the latter equation between successive parameters $s_n$ with $\theta(f_n)=0$, so the equations $(\Delta - \lambda_w)u=\theta$ and $\theta u=0$ have a solution $u$ for at most one $w$ in each interval $\text{Im}(s_n)\leq \text{Im}(w)\leq \text{Im}(s_{n+1})$.

For $\log\log t$ large, the spacing of the spectral parameters $s_n$ such that $a^{s_n}+c_{s_n}a^{1-s_n}=0$ is essentially regular. Namely, $\psi(t)=\arg \xi(1+2it)$ satisfies \cite{Tit86}
\[
\psi(t)=t\log t + O\left(\frac{t\log t}{\log\log t}\right)\hspace{5mm}\text{and}\hspace{5mm}\psi'(t)=\log t + O\left(\frac{\log t}{\log\log t}\right).
\]
Thus, given $\varepsilon>0$, there is a $t_o$ sufficiently large such that for $w_1,w_2\in\{\frac{1}{2}+it\in \C: t\geq t_o\}$ such that $\lambda_{w_1}$, $\lambda_{w_2}$ are eigenvalues of the Friedrichs extension $\tilde T_{\theta}$ of $\Delta$ restricted to test functions in the kernel of $\theta$,
\[
\left|\text{Im}(w_1)-\text{Im}(w_2)\right|\geq (1-\varepsilon)\cdot \frac{\pi}{\log t}.
\]

Finally, let
\[
J(w)=\frac{h_d^2}{-\lambda_w \cdot \langle1,1\rangle} + \frac{1}{4\pi i}\int_{(1/2)} \left|\frac{\zeta_k(s)}{\zeta(2s)}\right|^2 - \left|\frac{\zeta_k(w)}{\zeta(2w)}\right|^2 \frac{ds}{\lambda_s - \lambda_w},
\]
with $w=\frac{1}{2}+i\tau$. After re-arranging, the condition that $\lambda_w$ be an eigenvalue of $\tilde T_{a,\theta}$ (or $\tilde T_{\geq a,\theta}$) becomes
\[
\cos(\tau\log a +\psi(\tau))\cdot J(w)=\sin(\tau\log a +\psi(\tau))\cdot\frac{\theta E_{1-w}\cdot\theta E_w}{2\tau}.
\]
Between any two consecutive zeros of $\cos(\tau\log a +\psi(\tau))$, there is a unique $\tau$ such that $\lambda_{\frac{1}{2}+i\tau}$ is an eigenvalue of $\tilde T_{a,\theta}$ \cite{Tit86}. That is, {half} the on-line zeros $w_o$ of $J(w)$, if any, repel upward the on-line zeros $s_0$ of $\zeta_k(s)$, in the sense that, given $\varepsilon>0$, there is a sufficiently large $t_o$ such that above a zero $w_0$ of $J(w)$ with $\text{Im}(w_0)\geq t_o$, the next on-line zero $s_0$ of $\zeta_k(s)$ must satisfy
\[
\left|\text{Im}(s_0)-\text{Im}(w_0)\right| \geq (1/2 -\varepsilon)\cdot \text{average spacing} \geq (1/2 - \varepsilon)\cdot \frac{\pi}{\log t}.
\]

\section{Automorphic Hamiltonians}
\label{Results}
Building on \cite{Dir28}, \cite{Fri34}, \cite{Gro53}, \cite{LP76}*{pp. 204-206}, \cite{CdV82}, \cite{BG20}, and others, we construct a Hamiltonian which discretizes $L^2(\Gamma\backslash G/K)$ and will be shown to characterize a nuclear Fr\'echet automorphic Schwartz space. 

\subsection{Invariant Dirac operators}

Given an $r$-dimensional vector space $V$ over a field $K$ equipped with a quadratic form $q:V\to K$, the coefficients of a (suitable) Dirac operator are in the Clifford algebra. The Clifford algebra is a unital associative algebra $C\ell(V,q)$ and a linear map $i: V\to C \ell(V,q)$ characterized by the following universal property: given any unital associative algebra $A$ over $K$  and a linear map $j:V\to A$ such that for all $v\in V$,
\[
j(v)^2 =  -q(v)1_{A},
\]
there is a unique algebra homomorphism $J: C\ell(V,q)\to A$ such that the following diagram commutes:
\[
\begin{tikzcd}
C\ell(V,q) \arrow[dr,"\exists!\hspace{1mm}J",dashed]\\
V \arrow[u,"i"] \arrow[r, "\forall\hspace{1mm}j",swap] & A.
\end{tikzcd}
\]
Additionally, {as proposed in \cite{Che54} and \cite{Bou59}}, $C\ell(V,q)$ can be constructed as the quotient of the tensor algebra $T(V)$ over $V$ by the two-sided ideal generated by elements of the form $v\otimes v+ q(v)\cdot 1$.\\

To motivate examples, it is useful to note that the relation $j(v)^2=-q(v)1_A$ in $A$ imposes orthogonality relations which describe its structure in coordinates. Namely, for $\{e_i\}$ an orthonormal basis of $V$ embedded in $C\ell(V)$, the relations are
\[
e_i^2 = -1,\hspace{5mm}e_ie_j = -e_ie_j\hspace{3mm} (i\neq j).
\]
For example, with $q(v)=\norm{v}$ the usual Euclidean metric,
\begin{align*}
C\ell(\R,q) \approx \C, \hspace{5mm} C\ell(\R^2,q)\approx\mathbb{H}, 
\end{align*}
for $\mathbb{H}$ the Hamiltonian quaternions. Alternatively, since {all positive-definite quadratic forms on an $r$-dimensional complex vector space are equivalent to $q(z_1,...,z_r)=z_1^2+...+z_r^2$}, 
\begin{align*}
C\ell(\C,q) \approx \C\oplus \C, \hspace{5mm}C\ell(\C^2,q)\approx \C\otimes_{\R}\mathbb{H},
\end{align*}
where $\C\otimes_{\R}\mathbb{H}$ are the biquaternions. In what follows, we assume the standard positive-definite quadratic form and denote $C\ell(V,q)$ as $C\ell(V)$.\\

Next, we recall properties of invariant Dirac operators on Lie groups $G$. As in Section \ref{Laplacian}, we offer a coordinate-free characterization which facilitates the proof of invariance.\\

For $\mathfrak{g}={\mathfrak{p}}\oplus\mathfrak{k}$ a Cartan decomposition and $\rho$ the $K$-equivariant map given below, the invariant Dirac operator is $\mathbb{D}=\rho(\text{id}_{\mathfrak{p}})$.\\

\textbf{\underline{{Coordinate-free characterization of the invariant Dirac operator}}}\\

 \[
\begin{tikzcd}[column sep=large]
\End_{\R}(\mathfrak{p}) \arrow[r, " \approx"] \arrow[d,dash] \arrow[bend left, "\rho"]{rrr} & \mathfrak{p}\otimes\mathfrak{p}^*  \arrow[r, "\approx\hspace{1mm}via\hspace{1mm}\langle\rangle"]  & \mathfrak{p}\otimes\mathfrak{p} \arrow[r, "{incl}"] & C\ell(\mathfrak{p})\otimes U\mathfrak{g} \arrow[d,dash] \\
\text{id}_{\mathfrak{p}} \arrow[rrr, mapsto] & & & \rho(\text{id}_{\mathfrak{p}})=\mathbb{D}.
\end{tikzcd}
\]

This coordinate-free characterization allows us to express the invariant Dirac operator $\mathbb{D}=\rho(\text{id}_{\mathfrak{p}})$ in terms of any orthogonal basis $e_1,...,e_r$ of $\mathfrak{p}$ embedded in $C\ell(\mathfrak{p})$:  For $\lambda_1,...,\lambda_r$ the corresponding dual basis of the dual $\mathfrak{p}^*$, which is characterized by $\lambda_i(e_j)=\delta_{ij}$, and $X_1,...,X_r$ the dual basis for $\mathfrak{p}$ in terms of $\langle -,-\rangle$, characterized by $\langle e_i,X_j\rangle=\delta_{ij}$, we have
 \[
\begin{tikzcd}[column sep=large]
\End_{\R}(\mathfrak{p}) \arrow[r, " \approx"] \arrow[d,dash] \arrow[bend left, "\rho"]{rrr} & \mathfrak{p}\otimes\mathfrak{p}^* \arrow[d,dash] \arrow[r, "\approx\hspace{1mm}via\hspace{1mm}\langle\rangle"]  & \mathfrak{p}\otimes\mathfrak{p} \arrow[d,dash]\arrow[r, "{incl}"] & C\ell(\mathfrak{p})\otimes U\mathfrak{g} \arrow[d,dash] \\
\text{id}_{\mathfrak{p}}\arrow[r, mapsto]   & \sum_{i} e_i\otimes\lambda_i \arrow[r,mapsto]  & \sum_i e_i\otimes X_i \arrow[r,mapsto] & \sum_i e_i\otimes X_i=\mathbb{D}.
\end{tikzcd}
\]
Since $\rho$ is $K$-equivariant and $\text{id}_\mathfrak{p}$ is $K$-invariant, $\dirac=\rho(\text{Id}_\mathfrak{p})$ is $K$-invariant and descends to the quotient $G/K$. Furthermore, right differentiation by $\mathfrak{g}$ commutes with left multiplication by $G$ by associativity, so $\dirac$ descends to the further quotient $\Gamma\backslash G/K$ for $\Gamma\subseteq G$ discrete.\\

Furthermore, $\dirac$ has the expected property:
\begin{theorem}
For $\dirac$ the invariant Dirac operator, $\dirac^2 = -\Delta$ on $G/K$.
\end{theorem}
\begin{proof}
The coordinate-wise orthogonality relations of the Clifford algebra $\Cl(\mathfrak{p})$ yield
\begin{align*}
    \dirac^2&={\sum_i e_i^2\otimes X_i^2} + \sum_{i<j}\left(e_ie_j\otimes X_iX_j + e_je_i\otimes X_jX_i\right)\\
    &=-1\otimes\left(\Omega_\mathfrak{g}-\Omega_\mathfrak{k}\right)+\sum_{i<j}e_ie_j\otimes{[X_i,X_j]},
\end{align*}
where $\Omega_\mathfrak{g}=\Omega$ is the Casimir operator in $U\mathfrak{g}$. Since the Cartan involution is a Lie algebra automorphism, $[X_i,X_j]\in\mathfrak{k}$; furthermore, $\Omega_\mathfrak{k}$ acts by 0 on $G/K$. Thus $\dirac^2= -1\otimes\Omega_\mathfrak{g}=-\Delta$ on $G/K$.
\end{proof}

\subsection{A canonical confining potential}
Next, we motivate the choice of the simplest automorphic Hamiltonian.
\begin{theorem}[S.]
Let $G$ be a semi-simple Lie group with invariant Dirac operator $\dirac$. If $\beta$ acts by central scalars and $\Delta \beta\in \R$, then the Hamiltonian $S=-\Delta-(\dirac\beta)^2$ has ground state $f=e^{-\beta}$ and bottom eigenvalue $\Delta\beta$.
\end{theorem}
\begin{proof}
By design, $S$ factors into raising and lowering operators
\[
S=-\Delta+q = (\dirac-\dirac\beta)(\dirac+\dirac\beta)+[\dirac\beta,\dirac] =RL+[\dirac\beta,\dirac],
\]
and Leibniz' rule simplifies the commutator
\[
[\dirac\beta,\dirac]f=\dirac f\cdot \dirac\beta-\dirac f\cdot \dirac\beta-f\cdot \dirac^2\beta=f\cdot\Delta\beta.
\]
Since the commutator $[\dirac\beta,\dirac]=\Delta \beta$ is a real scalar, it is also the lowest possible eigenvalue of $S$ by mutual adjointness of the raising and lowering operators; furthermore, $S$ has ground state $f=e^{-\beta}$, since $Lf=0$.
\end{proof}

{\subsection{Kronecker's First Limit Formula}}
\label{KLF}
Largely following the discussion in \cite{Ter73}, \cite{Sie61}, \cite{DIT18}, and \cite{LM15}, we recall Kronecker's first limit formula and several immediately relevant generalizations. \cite{Sie61} elucidates the significance of Kronecker limit formulas in number theory, and \cite{Wei76} traces their historical development.

Introduced in \cite{Kro68}, Kronecker's first limit formula is a two-dimensional analogue of the limit formula for Riemann's $\zeta(s)$:
\[
\lim_{s\to 1}\left(\zeta(s)-\frac{1}{s-1}\right)=\gamma_o,
\]
where $\gamma_o$ is the Euler-Mascheroni constant. For $z \in\mathfrak{H}$, let $\displaystyle Q(u,v)={\frac{\text{Im }z}{|uz+v|^2}}$, and define
\[
\zeta_Q(s)=\sum_{0\neq (m,n)\in \Z^2} Q(m,n)^{-s}= {2}\zeta(2s)E_s(z), \hspace{5mm}\text{Re }s>1.
\]
Kronecker's first limit formula is
\[
\lim_{s\to 1}\left(\zeta_Q(s)-\frac{\pi}{s-1}\right)=2\pi(\gamma_o - \log 2 -\log(\sqrt y|\eta(z)|^2)),
\]
where $\eta(z)$ is Dedekind's eta function.  There are many proofs of this result, including Siegel's proof by Poisson summation \cite{Sie61}*{pp. 4-13}. In his 1903 paper, Epstein obtained an analogous formula \cite{Eps03}*{pp. 644} for his namesake zeta function
\[
Z_r(Q, s) = \sum_{0\neq v\in \Z^r}(v\hspace{1mm}Q\hspace{1mm}v^\top )^{-s}, \hspace{5mm}\text{Re}(s)>\frac{r}{2},
\]
for a positive-definite matrix $Q\in GL_r(\R)$ with Iwasawa decompositions
\[
Q=\begin{pmatrix}1_{\ell} & * \\ 0 & 1_{r-\ell} \end{pmatrix}\begin{pmatrix}A & 0 \\ 0 & D \end{pmatrix}\cdot k
\]
with $A\in GL_{\ell}(\R)$ and $D\in GL_{r-\ell}(\R)$.  As proven by Epstein for $\ell=1$ and for $1\leq \ell<r$ in \cite{Ter73}*{Theorem 4}, 
\[
\lim_{s\to r/2}\left(Z_r(Q,s)-\frac{\pi^{r/2}}{\sqrt{|Q|}\Gamma\left(\frac{r}{2}\right)(s-\frac{r}{2})}\right)= 2Z^*_{r/2}(Q),
\]
where the factor of $2$ on the right-hand side of the equality comes from our normalization of $Z_r(Q,s)$, 
\begin{align*}
Z^*_{r/2}(Q) &= Z_{r-\ell}\left(D,\frac{r}{2}\right) + \frac{\pi^{\frac{r-\ell}{2}}\Gamma \left(\frac{\ell}{2}\right)}{\sqrt{\det D}\cdot\Gamma \left(\frac{r}{2}\right)}\cdot Z^*_{\ell/2}(A)\\
&+ \frac{\pi^{\frac{r}{2}}}{\Gamma\left(\frac{r}{2}\right)} \cdot H_{\ell,r-\ell}\left(Q,\frac{r}{2}\right)+ \frac{\pi^{\frac{r}{2}}}{2\sqrt{\det Q}\cdot\Gamma\left(\frac{r}{2}\right)}\cdot\left(\frac{\Gamma'\left(\frac{\ell}{2}\right)}{\Gamma\left(\frac{\ell}{2}\right)} - \frac{\Gamma'\left(\frac{r}{2}\right)}{\Gamma\left(\frac{r}{2}\right)}\right),
\end{align*}
\[
H_{\ell,r-\ell}\left(Q,\frac{r}{2}\right) = \sum_{\substack{0\neq u\in \Z^{\ell} \\ 0\neq v\in \Z^{r-\ell}}} \frac{\exp^{(2\pi i \cdot vQu^\top)}}{\sqrt{\det D}}\left(\frac{vD^{-1}v^\top}{uAu^\top}\right)^{\frac{\ell}{4}}\\  K_{\frac{\ell}{2}}\left(2\pi\left(uAu^\top vD^{-1}v^{\top}\right)^{1/2}\right),
\]
and $K_{\nu}$ is the modified Bessel function of the second kind defined by
\[
K_{\nu}(z)=\frac{1}{2}\int_0^\infty\exp\left(-z(u+1/u)/2\right)u^{\nu-1}du,\hspace{5mm}\text{for }|\arg z|<\pi/2.
\]  
Terras' proof generalizes the proof of \cite{BG64}*{Theorem 1} for the case of binary quadratic forms.

We now discuss Kronecker limit formulas for Dedekind zeta functions of an ideal class $\mathfrak{a}$ of a number field $K$, defined by
\[
\zeta_K(s,\mathfrak{a}) = \sum_{\mathfrak{b}\sim \mathfrak{a}} (N {\mathfrak{b}})^{-s},\hspace{5mm}(\text{Re }s>1).
\]
Using the classical identity of Dirichlet/Hecke relating $\zeta_K(s,\mathfrak{a})$ for $K=\Q(\sqrt{-D})$ an imaginary quadratic field to the value of $E_s(z)$ at the {Heegner point} $\tau_{\mathfrak{a}}$ (see {\cite{Hec17a}} or e.g. \cite{Gar14}, \cite{Tem11}), 
\[
E_s(\tau_{\mathfrak{a}}) = \frac{\# K^\times}{4}\left(\frac{\sqrt{|-D|}}{2}\right)^s \frac{\zeta_K(s,\mathfrak{a})}{\zeta(2s)},
\]
Kronecker's first limit formula yields a limit formula for $\zeta_K(s,\mathfrak{a})$ of imaginary quadratic field. In \cite{Hec17b}, Hecke proved a similar limit formula for real quadratic fields by relating $\zeta_K(s,\mathfrak{a})$ to an integral of $E_s(z)$ over a {Heegner cycle} in $\mathfrak{H}$. Decades later, \cite{BG84} proved the limit formula for real cubic fields using an identity relating the integral of a minimal parabolic Eisenstein series on $SL_3(\R)$ over a Heegner cycle to the {Rankin-Selberg integral} of a {Hilbert modular Eisenstein series}, and \cite{Efr92} proved a similar result for non-real cubic fields by using a maximal parabolic Eisenstein series. More recently, \cite{LM15} generalize Efrat's method to identify an analogous limit formula for  $\zeta_K(s,\mathfrak{a})$ of a {wide ideal class $\mathfrak{a}$} of a totally real number field $K$ of degree $r\geq 2$. Their proof closely follows \cite{Efr92}, \cite{Fri87}, \cite{Gol15}, and \cite{Ter73}.

\subsection{Sufficient growth of the confining potential}
\label{Growth}
We now give a sufficient condition for a Hamiltonian $S=-\Delta+q$ to have purely discrete spectrum on $L^2(\Gamma\backslash G/K)$ and characterize a nuclear Fr\'echet Schwartz space. {We largely follow the discussion in \cite{Gar13}}. Let $V^n$ be the Hilbert-space completions of automorphic test functions $C_c^\infty(\Gamma\backslash G/K)$ under the Sobolev-like norms
\[
|f|^2_{V^n}=\langle S^nf, f\rangle_{L^2}.
\]
The confining potential $q$ has ``sufficient growth" when the $V^n$ have Hilbert-Schmidt transition maps. {This is convenient because, as discussed in Section \ref{Schwartz}, the projective limit of Hilbert spaces with Hilbert-Schmidt transition maps is nuclear Fr\'echet. Nuclearity of $V^\infty=\lim_n V^n$ guarantees existence of a genuine tensor product and, therefore, gives Schwartz' kernel theorem, by the Cartan-Eilenberg adjunction}. As one might imagine, nuclearity requires more than compactness of the inclusion, and more growth of the potential is required as $\dim_{\R}(\Gamma\backslash G/K)$ increases.

\cite{Gar13} suffices for the simplest case:
{
\begin{theorem}
Let $\Gamma\backslash G/K=SL_2(\Z)\backslash SL_2(\R)/SO(2,\R)$, and let $V^1$ be the Hilbert-space completion of $C^\infty_c(\Gamma\backslash G/K)$ under the Sobolev-like norm
\[
|f|^2_{V^1}= \langle Sf, f\rangle_{L^2}.
\] The inclusion $V^1\to L^2(\Gamma\backslash G/K)$ is compact for any $q(x+iy)\gg y^{\varepsilon}$ for $\varepsilon>0$.
\end{theorem}}

\begin{proof}
The proof is an easier variant of the compactness argument in \cite{LP76}*{pp. 204-206}. By the total boundedness criterion for relative compactness, it suffices to show that the image of the unit $V^1$-ball $B$ is totally bounded in $L^2(\Gamma\backslash G/K)$. That is, we want to show that given $\varepsilon>0$, the image of $B$ in $L^2(\Gamma\backslash G/K)$ can be covered by finitely-many balls of radius $\varepsilon$. The usual {Rellich-Kondrachev lemma on compact Riemannian manifolds} asserts compactness of the standard Sobolev spaces {$H^{k+1}(\T^n)\to H^k(\T^n)$}, so the issue reduces to an estimate on the {tail} of $\Gamma\backslash G/K$:

Given $c\geq 1$, let $Y_o$ be the image of $\frac{\sqrt3}{2}\leq y\leq c+1$ in $\Gamma\backslash G/K$ and $Y_\infty$ the image of $y\geq c$. We may cover $Y_\infty$ by one large open set $U_\infty$ and $Y_o$ by small coordinate patches $U_i$. By compactness of $Y_o$, there exists a finite sub-cover of $Y_o$. Choose a smooth partition of unity $\{\varphi_i\}$ subordinate to the finite sub-cover and $U_\infty$ , and let $\varphi_\infty$ be a smooth function identically 1 on the image of $y\geq c+1$. On $Y_o$, we may write $f\in V^1$ as a finite sum of functions $\varphi_i\cdot f$ with compact support on the $U_i$ identified with $\T^2$, and lying in $H^1(\T^2)$ since
\[
|\varphi_i f|^2_{V^1}=\langle (-\Delta + q)\varphi_i f,\varphi_i f\rangle \ll_i \langle(-\Delta^{E}+1)\varphi_i f, \varphi_i f\rangle=|\varphi_i f|^2_{H^1(\T^2)},
\]
for {$\Delta^E$ the Euclidean Laplacian on $\T^2$}. By the usual Rellich lemma, each copy of the inclusion map $H^1(\T^2)\to L^2(\T^2)$ is compact, so $\varphi_i \cdot B$ is totally bounded in $L^2(\Gamma\backslash G/K)$.

It remains to prove that given $\delta >0$, the cut-off $c$ can be made sufficiently large such that the image of the smoothed {tail} $\varphi_\infty\cdot B$ lies in a single ball of radius $\delta$.  Since $0\leq \varphi_\infty\leq 1$, it suffices to show 
\[
\lim_{c\to\infty}\int_{y>c}|f(z)|^2\frac{dxdy}{y^2}\to 0\hspace{10mm}(\text{uniformly for }|f|_{V^1}\leq 1).
\]
Since $q\gg y^{\varepsilon}$ in $Y_\infty$,
\begin{align*}
\int_{y>c}|f(z)|^2\frac{dxdy}{y^2}\leq c^{-\varepsilon}\int_{y>c}|f(z)|^2y^\varepsilon \frac{dxdy}{y^2}\leq c^{-\varepsilon}|f|^2_{V^1}\leq c^{-\varepsilon}\to 0
\end{align*}
as $c\to \infty$. Thus, the image of $B$ in $L^2(\Gamma\backslash G/K)$ is totally bounded, and by the total boundedness criterion for relative compactness, the inclusion $V^1\to L^2(\Gamma\backslash G/K)$ is compact.
\end{proof}

Next, we identify a growth condition such that the transition maps $V^n \to V^{n-2}$ are Hilbert-Schmidt. Certainly, the transition maps are Hilbert-Schmidt if $-\Delta+q$ has Hilbert-Schmidt resolvent: for  $\{u_i\}$ an orthonormal basis of eigenfunctions for $L^2(\Gamma\backslash G/K)$, with positive eigenvalues $\lambda_i$, the vectors $u_i/\lambda_i^{n/2}$ form an orthonormal basis for $V^n$. With respect to these bases, the inclusions $V^n\to V^{n-2}$ are multiplication maps
\[
\sum_i c_i \frac{u_i}{\lambda_i^{n/2}} \mapsto \sum_i c_i \cdot \lambda_i^{-1}\frac{u_i}{\lambda_i^{(n-2)/2}},
\]
and such a map is Hilbert-Schmidt if and only if $\sum_i (\lambda_i^{-1})^2 < \infty$. The resolvent of the Friedrichs extension of $-\Delta+q$ has eigenvalues $\lambda_i^{-1}$, and the Hilbert-Schmidt property is the same equality.

As in \cite{Gar13}, we can identify a growth condition for Hilbert-Schmidt resolvent in the simplest case: 
\begin{theorem}
\label{HSGrowth}
Let $\Gamma\backslash G/K=SL_2(\Z)\backslash SL_2(\R)/SO(2,\R)$, and let $V^n$ be the Hilbert-space completions of $C^\infty_c(\Gamma\backslash G/K)$ under the Sobolev-like norm
\[
|f|^2_{V^n}= \langle S^nf, f\rangle_{L^2}.
\] 
The Hamiltonian $-\Delta + q$ has Hilbert-Schmidt resolvent when $q(x+iy)\gg y^2$.
\end{theorem}
\begin{proof}
As in the proof of compactness of the inclusion $V^1\to L^2(\Gamma\backslash G/K)$, the fact that $H^{{k+2}}(\T^2)\to H^{{k}}(\T^2)$ is Hilbert-Schmidt reduces the discussion to the non-compact but geometrically simpler part of $\Gamma\backslash G/K$.

Specifically, it suffices to consider the restriction $T$ of $S=-\Delta + y^2$ to test functions on the tapering cylinder $X=\T^1\times[1,\infty)$ with measure $\frac{dxdy}{y^2}$. Thus, the domain of $T$ includes test functions on $X$ vanishing to infinite order on the boundary $\partial X=\T^1 \times \{1\}$. Let $\tilde T$ be the Friedrichs self-adjoint extension of $T$.

The circle group $\T^1$ acts on $X$ and {commutes with $T$ and $\tilde T$}. Thus, $L^2(X)$ decomposes orthogonally into components indexed by characters $\psi_n(x)=e^{inx}$ of $\T^1$. The differential equation for the $n^{th}$ component of a fundamental solution $u_a=u_a(y)$ at {$y=a$} is given by
\[
\delta_a = (-\Delta+y^2) (e^{inx}u_a(y))= y^2\left(n^2 -\frac{\partial^2}{\partial y^2}+{1}\right)u_a(y)=y^2\left(-u_a'' + (n^2+{1})u_a\right) 
\]
and simplifies to a (nonhomogeneous) constant-coefficient equation
\[
\frac{\delta_a}{{a}^2} = -u_a'' + (n^2+{1})u_a,\hspace{5mm}(\text{with }a>1).
\]
As usual, we attempt to piece together $u_a$ from solutions $e^{\pm cy}$ to the associated homogeneous equation $-u_a'' +(n^2+{1})u = 0$, where $c=\sqrt{n^2+{1}}\geq 1$. Since $u_a(y)$ must have moderate growth enough as $y\to +\infty$ such that it is in $L^2(X)$ with measure $\frac{dxdy}{y^2}$, go to zero as $y\to 1^+$, be continuous, and be non-smooth enough at $y=a$ to produce the required multiple of $\delta_a$, $u_a$ must be of the form
\[
u_a(y) = \begin{cases} 
A_ae^{cy}+B_ae^{-cy}\hspace{5mm}(\text{for }1<y<a)\\ 
C_a e^{-cy}\hspace{20mm}(\text{for }a<y)
\end{cases}
\]
for some constants $A_a, B_a,C_a$ satisfying the following conditions:
\[
\begin{cases} 
A_ae^c + B_ae^{-c}=0\hspace{48mm}(\text{vanishing at  }y\to 1^+)\\ 
A_ae^{ca} + B_ae^{-ca}=C_ae^{-ca}\hspace{35mm}(\text{continuity at  }y=a)\\
-cC_ae^{-ca} - (cA_ae^{ca}- cB_ae^{-ca})=\frac{1}{a^2}\hspace{14mm}(\text{change of slope by }\frac{1}{a^2}\text{ at } y=a)
\end{cases}
\]
From the first equation, $B_a = -e^{2c}\cdot A_a$ such that the system becomes
\[
\begin{cases} 
A_a(e^{ca} - e^{2c} e^{-ca})=C_ae^{-ca}\hspace{31mm}(\text{continuity at  }y=a)\\
-cC_ae^{-ca} - cA_a(e^{ca}+e^{2c}e^{-ca})=\frac{1}{a^2}\hspace{15mm}(\text{change of slope by }\frac{1}{a^2}\text{ at } y=a).
\end{cases}
\]
Substituting $C_a = A_a \cdot (e^{2ca} - e^{2c})$ from the first equation into the second yields the condition
\begin{align*}
\frac{1}{a^2}&=A_a\cdot\left(-c(e^{2ca} - e^{2c})e^{-ca} - c(e^{ca}+e^{2c}e^{-ca}\right)= A_a \cdot (-2ce^{ca}).
\end{align*}
Rearranging, we have $A_a= -e^{-ca}/2ca^2$ such that
\begin{align*}
B_a &= -e^{2c}\cdot A_a = e^{2c-ca}/2ca^2\\
C_a &= A_a \cdot (e^{2ca} - e^{2c})= (e^{2c}e^{-ca}-e^{ca})/2ca^2.
\end{align*}
Thus, 
\[
u_a(y) = \begin{dcases} 
\frac{e^{2c}e^{-cy}-e^{cy}}{2ca^2}\cdot e^{-ca}\hspace{5mm}(\text{for }1<y<a)\\ 
 \frac{e^{2c}e^{-ca}-e^{ca}}{2ca^2}\cdot e^{-cy} \hspace{5mm}(\text{for }a<y).
\end{dcases}
\]
{Since $(c^2-\frac{\partial^2}{\partial y^2})u = f$ is solved by
\[
u(y)= \int_1^{\infty} a^2 \cdot u_a(y) f(a) da,
\]
the symmetry of $-\Delta+q$ with respect to the measure $dy/y^2$ implies that $a^2\cdot u_a(y)$ is symmetric in $y$ and $a$}; furthermore, the integral kernel for the resolvent is
\[
a^2\cdot u_a(y) = \begin{dcases} 
\frac{e^{2c}e^{-cy}-e^{cy}}{2c}\cdot e^{-ca}\hspace{5mm}(\text{for }1<y<a)\\ 
 \frac{e^{2c}e^{-ca}-e^{ca}}{2c}\cdot e^{-cy} \hspace{5mm}(\text{for }a<y).
\end{dcases}
\]
The resolvent being Hilbert-Schmidt is equivalent to 
\[
\int_1^{\infty}\int_1^{\infty}|a^2\cdot u_a(y)|^2\hspace{1mm} \frac{da}{a^2}\hspace{1mm}\frac{dy}{y^2} < \infty.
\]
By symmetry, it suffices to integrate over $1<y<a<\infty$, and 
\begin{align*}
\iint_{1<y<a}|a^2\cdot u_a(y)|^2\hspace{1mm} \frac{da}{a^2}\hspace{1mm}\frac{dy}{y^2}=\iint_{1<y<a}\frac{|e^{4c}e^{-2cy}- 2e^{2c}+e^{2cy}|\cdot e^{-2ca}}{4c^2}\hspace{1mm} \frac{da}{a^2}\hspace{1mm}\frac{dy}{y^2}\\
\ll \frac{1}{n^2+1} \iint_{1<y<a}(e^{4c}e^{-2cy}+ 2e^{2c}+e^{2cy})\cdot e^{-2ca}\hspace{1mm} \frac{da}{a^2}\hspace{1mm}\frac{dy}{y^2}\
\end{align*}
Replacing $y,a$ by $y+1,a+1$, the integral becomes
\begin{align*}
\iint_{0<y<a}(e^{-2cy}+ 2+e^{2cy})&\cdot e^{-2ca}\hspace{1mm} \frac{da}{(a+1)^2}\hspace{1mm}\frac{dy}{(y+1)^2}\ll \iint_{0<y<a} \frac{da}{(a+1)^2}\hspace{1mm}\frac{dy}{(y+1)^2}\\
&\ll \int_0^\infty \frac{da}{(a+1)^2}\cdot \int_{0}^\infty\frac{dy}{(y+1)^2} < \infty
\end{align*}
Since the $L^2$-norm of the $n^{th}$ component of the integral kernel is uniformly bounded by a constant multiple of $1/(n^2+1)$, the sum over $n\in \Z$ is finite. That is, the resolvent is Hilbert-Schmidt.
\end{proof}
{One should identify an analogous growth condition for $G=SL_r(\R)$. }

\subsection{The simplest automorphic Hamiltonian}
The previous discussion leads to the following results.

%% SIMPLEST CASE
\begin{theorem}[S.]
Let $\Gamma\backslash G/K=SL_2(\Z)\backslash SL_2(\R)/SO(2,\R)$. For $\dirac$ the invariant Dirac operator and $E_{1}^*$ the zero-order term of the Laurent expansion of the simplest (degenerate) Eisenstein series $E_{s}$ at $s=1$ with eigenvalue $\lambda_s=2s(s-1)$, the Friedrichs extension of the Hamiltonian $S=-\Delta-(\dirac E_1^*)^2$ has purely discrete spectrum on $L^2(\Gamma\backslash G/K)$, ground state $f=e^{-E_1^*}$, and lowest eigenvalue {$\lambda_o={6/\pi}$}. Furthermore, $S$ {characterizes a nuclear Fr\'echet Schwartz space} $V^\infty=\lim_{n}V^n$, for $V^n$ the completion of $L^2(\Gamma\backslash G/K)$ with respect to the Sobolev norm
\[
|f|_{V^n}=\langle S^nf,f\rangle_{L^2}.
\]
\end{theorem}
\begin{proof}
For $x+iy\in\mathfrak{H}$ and $j, k$ the usual Hamiltonian quaternions, $\dirac=y(j\partial_x + k\partial_y)$ with
\[
\dirac^2=y^2\left(j^2\partial_x^2 +(jk+kj)\partial_x\partial_y + k^2\partial_y^2\right)=-y^2\left(\partial_x^2+\partial_y^2\right)=-\Delta.
\]
The Laurent expansion of $E_s$ about $s=1$ is
\[
E_s(z)=\frac{3/\pi}{(s-1)}+E_1^*(z)+O(s-1),
\]
with {$\displaystyle E_1^*(z)=\frac{6}{\pi}\left(\gamma-\log2 - \log\left(y^{1/2}|\eta(z)|^2\right)\right)$} given by Kronecker's first limit formula.

In what follows, we view $s\to E_s$ as a meromorphic {function-valued function} of $s$. By the (vector-valued) Laurent expansion of $s\to E_s$,
\begin{align*}
0&=\frac{d}{ds}[(\Delta-\lambda_s)(s-1)E_s]
= (\Delta-\lambda_s)\frac{d}{ds}[(s-1)E_s]- 2(2s-1)(s-1)E_s\\
&= (\Delta-\lambda_s)\frac{d}{ds}[(s-1)E_s]-2(2s-1)[3/\pi+(s-1)E_1^*+O(s-1)].
\end{align*}
Evaluating the expression at $s=1$ yields $\Delta E_1^*= {6}/\pi\in \R$.

We claim the confining potential $q=-(\dirac E_1^*)^2$ has sufficient growth (see Section \ref{Growth}). By Theorem \ref{HSGrowth}, we must show $q\gg y^2$ as $y\to +\infty$. Let $\mathfrak{g}$ be the Lie algebra of $G$ and $\mathfrak{z}$ the center of the universal enveloping algebra $U\mathfrak{g}$, generated by Casimir. Since $\Delta E_1^*\in \R$ is an eigenfunction of the Laplacian, $E_1^*$ and its derivatives are $\mathfrak{z}$-finite of moderate growth. By the theory of the constant term \cite{Gar18}*{Theorem 8.1.1}, it suffices to show $-c_P (\dirac E_1^*)^2 \gg y^2$, where
\[
\text{constant term of }f=c_Pf=\int_{(N\cap \Gamma)\backslash N}f(nx)dn.
\]
By left $G$-invariance and Gelfand-Pettis corollaries (see \cite{Gro53} and {\cite{Gar18}}), $\dirac$ commutes with evaluation of the constant term $c_P$ such that
\[
c_P(\dirac  E_s)=\dirac(c_PE_s)=\dirac\left(y^s+\frac{\xi(2s-1)}{\xi(2s)}\hspace{1mm}y^{1-s}\right)=k\left(sy^{s}+\frac{\xi(2s-1)(1-s)}{\xi(2s)}\hspace{1mm}y^{1-s}\right).
\]
Evaluating the Laurent expansion of $E_s$ at $s=1$ for large $y$,
\[
c_P(\dirac  E_1^*)=k\left(y-\frac{\text{Res}_{s=1}\xi(s)}{\xi(2)}\right)=k\left(y-\frac{3}{2\pi}\right)\gg ky,
\]
so
\[
-c_P(\dirac E_1^*)^2\gg -(ky)^2=y^2.
\]
Since $q=-(\dirac E_1^*)^2$ has sufficient growth, the Friedrichs extension of $S=-\Delta + q$ has purely discrete spectrum on $L^2(\modcurve)$, and for $V^n$ the completion of automorphic test functions with respect to the Sobolev norm
\[
|f|_{V^n}=\langle S^nf, f\rangle_{L^2},
\]
the projective limit $V^\infty=\lim_n V^n$ is nuclear Fr\'echet. $S$ also factors into mutually adjoint raising and lowering operators, $R=(\dirac -\dirac E_1^*)$ and $L=(\dirac +\dirac E_1^*)$: for $f$ in $L^2(\modcurve)$,
\[
 Sf=(\dirac -\dirac E_1^*)(\dirac + \dirac E_1^*)f + [\dirac E_1^*, \dirac]f=RLf + [\dirac E_1^*, \dirac]f.
\]
By Leibniz' rule, the commutator is $\Delta E_1^*=6/\pi$, and since $f=e^{-E_1^*}\in \ker L$,
\[
Sf = RLf+ [\dirac E_1^*, \dirac]f= RLf + \Delta E_1^*\hspace{1mm}f = \Delta E_1^* f.
\]
Thus, $f=e^{-E_1^*}$ is the ground state of $\tilde S$ with eigenvalue $\Delta E_1^*=6/\pi$.
\end{proof}

%% GENERALIZATION
\begin{theorem}[S.]
Let $\Gamma\backslash G/K=SL_r(\Z)\backslash SL_r(\R)/SO(r,\R)$ and $P=P^{r-1,1}$ the standard maximal proper parabolic subgroup of $G$. For $\dirac$ the invariant Dirac operator, $\varphi^P_s$ the spherical vector in the principal series representation of $G$, $E^P_{s,\varphi}$ the corresponding degenerate Eisenstein series with eigenvalue $\lambda_s=r(r-1)s(s-1)$, and $E_{1}^*$ the zero-order term of the Laurent expansion of $E^P_{s,\varphi}$ at {$s=1$}, the Hamiltonian $S=-\Delta-(\dirac E_1^*)^2$ has ground state $f=e^{-E_1^*}$ and lowest eigenvalue 
\[
\displaystyle\lambda_o=\frac{r(r-1)\pi^{r/2}}{2\Gamma(r/2)\zeta(r)}.
\] 
Granting sufficient growth of the confining potential, the Friedrichs extension $\tilde S$ has purely discrete spectrum on $L^2(\Gamma\backslash G/K)$ and {characterizes a nuclear Fr\'echet Schwartz space} $V^\infty=\lim_{n}V^n$, for $V^n$ the completion of $L^2(\Gamma\backslash G/K)$ with respect to the Sobolev norm
\[
|f|_{V^n}=\langle S^nf,f\rangle_{L^2}.
\]
\end{theorem}
\begin{proof}
By design, the proof is similar to the $SL_2(\R)$ case. The Laurent expansion of $E^P_s(g)$ about $s=1$ is 
\[
E^P_{s,\varphi}(g)=\frac{R}{(s-1)}+E_1^*(g)+O(s-1),
\]
with $\displaystyle R =\frac{\pi^{r/2}}{2\Gamma(r/2)\zeta(r)}$ the residue of $E_s^P(g)$ at $s=1$ and $E_{1}^{*}(g)$ given by a generalization of Kronecker's first limit formula.

In what follows, we view $s\to E^P_{s,\varphi}$ as a meromorphic {function-valued function} of $s$. By the (vector-valued) Laurent expansion of $s\to E^P_{s,\varphi}$, 
\begin{align*}
0&=\frac{d}{ds}[(\Delta-\lambda_s)\left(s-1\right)E^P_{s,\varphi}]
= (\Delta-\lambda_s)\frac{d}{ds}[\left(s-1\right)E^P_{s,\varphi}]- r(r-1)(2s-1)(s-1)E^P_{s,\varphi}\\
&= (\Delta-\lambda_s)\frac{d}{ds}[(s-1)E^P_{s,\varphi}]-r(r-1)(2s-1)[R+(s-1)E_1^*+O(s-1)].
\end{align*}
Evaluating the expression at $s=1$ yields $\Delta E_1^*= r(r-1)R=\lambda_o$.

Granting that $q=-(\dirac E_1^*)^2$ has sufficient growth, the Friedrichs extension of $S=-\Delta + q$ has purely discrete spectrum on $L^2(\modcurve)$ and for $V^n$ the completion of automorphic test functions with respect to the Sobolev norm
\[
|f|_{V^n}=\langle S^nf, f\rangle_{L^2},
\]
the projective limit $V^\infty=\lim_n V^n$ is nuclear Fr\'echet.

$S$ also factors into mutually adjoint raising and lowering operators, $R=(\dirac -\dirac E_1^*)$ and $L=(\dirac +\dirac E_1^*)$: for $f$ in $L^2(\modcurve)$,
\[
 Sf=(\dirac -\dirac E_1^*)(\dirac + \dirac E_1^*)f + [\dirac E_1^*, \dirac]f=RLf + [\dirac E_1^*, \dirac]f.
\]
By Leibniz' rule, the commutator is $\Delta E_1^*=\lambda_o$. Since $f=e^{-E_1^*}\in \ker L$,
\[
Sf = RLf+ [\dirac E_1^*, \dirac]f= RLf +\Delta E_1^*f = \Delta E_1^*f.
\]
That is, $f=e^{-E_1^*}$ is the ground state of $\tilde S$ with bottom eigenvalue $\displaystyle\lambda_o=\frac{r(r-1)\pi^{r/2}}{2\Gamma(r/2)\zeta(r)}$.
\end{proof}

\bibliographystyle{amsplain}

\end{document}